\documentclass[11pt, letterpaper]{article}
\usepackage{bm}
\usepackage{amsmath,amsthm,verbatim,amssymb,amsfonts,amscd, graphicx, graphics,physics}
\usepackage[mathscr]{eucal}
\usepackage{indentfirst}
\usepackage{booktabs}
\usepackage{pict2e}
\usepackage{epic}
\numberwithin{equation}{section}
\usepackage[footnotesize,bf]{caption}
\usepackage[a4paper, textwidth=155mm,top=27mm,bottom=29mm]{geometry}
\usepackage{epstopdf}
\usepackage{mathtools}
\usepackage[dvipsnames]{xcolor}
\usepackage{algorithmic}
\usepackage[ruled,vlined]{algorithm2e}
\usepackage{authblk}
\usepackage{placeins}

\theoremstyle{plain}
\newtheorem{theorem}{Theorem}[section]
\newtheorem{lemma}{Lemma}
\newtheorem{proposition}{Proposition}

\theoremstyle{definition}
\newtheorem{definition}{Definition}
\newtheorem{assumption}{Assumption}
\newtheorem{remark}[theorem]{Remark}

\usepackage{threeparttable}
\usepackage{subcaption} 
\usepackage{siunitx}
\usepackage{multirow}

\usepackage[backend=bibtex,style=alphabetic, firstinits=true,maxbibnames=9,natbib=true,url=false,sorting=nyt,doi=true,backref=false]{biblatex}
\renewbibmacro{in:}{\ifentrytype{article}{}{\printtext{\bibstring{in}\intitlepunct}}}

\usepackage[colorlinks,linkcolor =brown, citecolor=blue]{hyperref} 

\usepackage{cleveref}

\def\R{{\mathbb R}}
\def\rr{{\mathscr{R}}}

\def\P{{\mathbb P}}

\def\M{{\mathcal M}}

\def\G{{\mathcal G}}
\def\HH{{\mathcal H}}

\def\lll{{\widetilde{l}}}

\def\med{{\mathrm{med}}}

\def\sp{{\mathsf{s}}}
\def\uuu{{\widetilde{\bm u}}}
\def\uu{{\widehat{\bm u}}}
\def\vv{{\widetilde{\bm v}}}

\def\supp{{\mathrm{supp}}}

\renewcommand{\bar}{\overline}

\newcommand{\E}{\mathbb{E}}

\renewcommand{\d}[1]{\ensuremath{\operatorname{d}\!{#1}}}
\DeclareMathOperator{\sgn}{sgn}
\DeclareMathOperator{\diag}{diag}

\DeclareMathOperator*{\argmax}{arg\,max}

\usepackage{setspace}

\begin{document}

\title{Bradley--Terry model under general comparison graphs}

\date{\today}

\setstretch{1.18}

\author[1]{Ruijian Han}
\author[1]{Yuanhang Luo}
\author[2]{Yiming Xu}
\affil[1]{Department of Data Science and Artificial Intelligence, The Hong Kong Polytechnic
University, Hung Hom, Kowloon, Hong Kong SAR}
\affil[2]{Department of Mathematics, University of Kentucky, Lexington, KY, 40506, USA}

\maketitle

\begin{abstract}
The Bradley--Terry model is a parametric model for ranking from pairwise comparisons. Existing asymptotic theory for the maximum likelihood estimator (MLE) in regimes where the number of objects grows often requires homogeneity assumptions or compatibility conditions on comparison graphs, which limits its applicability to many practical settings. In this work, we establish uniform consistency of the MLE under general deterministic comparison designs. Our pairwise error bound consists of a pair-specific term governed by the effective resistance between the corresponding objects and a global term governed by the spectral gap of the unnormalized graph Laplacian. The bound is tight up to logarithmic factors for certain graph sequences. For any prespecified sequence of pairs satisfying an additional balancing condition, we further establish asymptotic normality of the corresponding estimated utility differences. As an application, we obtain uniform consistency for independent-edge random graph designs whenever the minimum edge probability exceeds the Erd\H{o}s--R\'enyi connectivity threshold by a logarithmic factor, as well as asymptotic normality under additional balancing conditions.
\end{abstract}


\section{Introduction}

Ranking objects from pairwise comparisons is a common task in sports analytics~\citep{baker2014dynamic, bozoki2016application}, psychometrics~\cite{thurstone1927law}, economics~\cite{mcfadden1973conditional}, and machine learning~\cite{Christiano2017, Rafailov2023, Sun2025}. The Bradley--Terry (BT) model is a classical approach to aggregating pairwise comparison outcomes into a global ranking. First introduced by \citet{MR1545015} and later formulated by \citet{MR0070925}, the model associates each object $i$ with a positive strength parameter $\pi_i=e^{u_i}>0$. For a comparison between objects $i$ and $j$, the probability that $i$ is preferred to $j$ is given by
\begin{align*}
\P(i\succ j)
=\frac{\pi_i}{\pi_i+\pi_j}=
\sigma(u_i-u_j),
\qquad
\sigma(u)\coloneqq (1+e^{-u})^{-1}.
\end{align*}
Here, $i\succ j$ denotes that $i$ is preferred to $j$, with the interpretation of ``preference'' depending on the application. The strengths $\pi_i$, or equivalently the utilities $u_i$, induce a global ranking of the objects. In the remainder of the paper, we use the utility parametrization. For further background on the BT model, see \cite{hamilton2025many}.

This work studies the asymptotic theory of the MLE in the BT model. Given independent pairwise comparison outcomes among $n$ objects, the MLE maximizes the joint likelihood conditional on the comparison graph, subject to an identifiability constraint. This constraint is necessary because the BT model is invariant under a common shift of all utilities. Our main question is under what conditions the MLE uniformly consistently estimates the underlying utilities as $n\to\infty$. Throughout, we denote the ground-truth utility vector and the MLE by $\bm u^*=(u_1^*,\ldots,u_n^*)^\top$ and $\uu=(\widehat{u}_1,\ldots,\widehat{u}_n)^\top$, respectively, and impose the centering constraint $\bm 1^\top\bm u^*=\bm 1^\top\uu=0$.

For fixed $n$, the MLE is consistent if the graph is connected and the minimum number of comparisons over its edges diverges; see, e.g., \cite[Theorem 5.1]{han2026convergence}. In many modern applications, however, the number of objects is large and more naturally treated as increasing. This motivates a high-dimensional regime in which both the number of objects and the volume of comparisons tend to infinity. To assess whether all utilities can be estimated simultaneously, we study the asymptotic theory of the MLE under the uniform error, namely, $\|\uu-\bm u^*\|_\infty$. This large-$n$ regime has attracted substantial interest in the recent study of the BT model \cite{fang2026recent}.

\paragraph{Related work.}
The uniform consistency of the MLE in the BT model depends critically on the comparison graph. \citet{MR1724040} established uniform consistency for the complete comparison graph, which was subsequently extended to certain densely connected graphs \cite{MR2987494}. \citet{han2020asymptotic} further advanced the theory for sparse designs, establishing uniform consistency when the comparison graph is sampled from the Erd\H{o}s--R\'enyi model $G(n,p_n)$ with $np_n\geq(\log n)^3$. Under the same random graph model, \citet{chen2022partial} established uniform error bounds under the optimal graph sparsity condition $np_n\geq C\log n$ for some absolute constant $C>0$, matching the order of the Erd\H{o}s--R\'enyi connectivity threshold \cite{MR0125031}. Their analysis builds on the leave-one-out perturbation approach developed in \cite{MR3953449}.

The Erd\H{o}s--R\'enyi model assumes homogeneous edge probabilities, which limits its applicability for comparison networks with substantial heterogeneity or community structure. Establishing uniform consistency under more general comparison designs is of practical relevance yet remains challenging. \citet{pmlr-v180-li22g} derived uniform error bounds for the regularized MLE in terms of the spectral gap of the unnormalized graph Laplacian and the maximum degree ratio. \citet{han2023general} established uniform consistency following a graph-based chaining idea. For stochastic block models (SBMs) \cite{holland1983stochastic} with a fixed number of blocks and a bounded score dynamic range, \citet{han2023general} established uniform consistency when the minimum edge probability is of order $\omega(n^{-1}(\log n)^3)$. Beyond the random graph models treated explicitly in their work, however, verifying certain abstract graph conditions and deriving concrete error rates for a given heterogeneous comparison design remain elusive.
 
From an electrical-network perspective, \citet{chen2023ranking} related pairwise MLE errors to effective resistances in a weighted comparison graph whose Laplacian is the Hessian of the negative log-likelihood evaluated at the true utility vector, following earlier resistance-based analyses \cite{hendrickx2019graph,hendrickx2020minimax}. Their general theorem, however, requires additional compatibility conditions involving auxiliary graph-dependent quantities, which can be difficult to verify for arbitrary comparison designs. Building on this, \citet{yang2024top} obtained more explicit spectral conditions for a weighted MLE, which also yield guarantees for the unweighted MLE when all likelihood weights are set to one. These conditions involve the maximum weighted degree and can be restrictive for imbalanced comparison designs. For general comparison designs, asymptotic theory has also been studied under other error metrics \cite{MR3504618, bong2022generalized} or alternative spectral estimators \cite{fan2023spectral,lee2026entrywise}.

Beyond uniform consistency, the asymptotic normality of the MLE has also been studied, including for the complete graph \cite{MR1724040}, polynomially sparse Erd\H{o}s--R\'enyi graphs \cite{han2020asymptotic}, and Erd\H{o}s--R\'enyi graphs at the optimal sparsity level \cite{gao2023uncertainty}. These works rely on the homogeneity of the underlying comparison graph to carry out fine perturbation arguments. \citet{han2024statistical} further established asymptotic normality for a general class of pairwise comparison graph models, including the BT model as a special case, while allowing mild heterogeneity among the compared objects. As empirically observed by \citet{han2023unified}, a general asymptotic normality result need not hold when the comparison graph is severely imbalanced, even if it is highly connected.

\paragraph{Main results.}

We introduce a new approach to establishing uniform consistency of the MLE under general deterministic comparison designs $\G_n$. 
An informal version of our main result \Cref{thm:main} states that, under suitable conditions (e.g., bounded dynamic range) and assuming $\lambda_2(\bm L_{\G_n})=\omega((\log n)^2)$, with probability at least $1-n^{-2}$, simultaneously for all $1\leq i<j\leq n$,
\begin{align}\label{bdd-intro}
\left|(\widehat{u}_i-\widehat{u}_j)-(u_i^*-u_j^*)\right|&\lesssim
\sqrt{\rr_{\G_n}(i\leftrightarrow j)\log n}
+\frac{(\log n)^2}{\lambda_2(\bm L_{\G_n})},
\end{align}
where $\rr_{\G_n}(i\leftrightarrow j)$ denotes the effective resistance between vertices $i$ and $j$ in the comparison graph $\G_n$, and $\lambda_2(\bm L_{\G_n})$ denotes the spectral gap of its unnormalized Laplacian. 

The first term in \eqref{bdd-intro} arises from linearization and can be viewed as a pair-specific measure of local estimation difficulty, while the second is a remainder term controlled by the global connectivity of the graph. To our knowledge, this is the first uniform pairwise error bound for the MLE under general deterministic comparison designs that combines a pair-specific effective-resistance term with an explicit spectral-gap remainder, without additional degree-imbalance restrictions or auxiliary compatibility conditions. Moreover, the bound is tight up to logarithmic factors for certain graph sequences; see \Cref{thm:bipartite}.

As an application of \eqref{bdd-intro}, consider a random comparison design in which each pair $\{i,j\}$ is selected independently with probability $p_{ij}$, and each selected pair is compared at least once. Conditional on the design, all comparison outcomes are independent. Let $p_{\min}=\min_{1\leq i<j\leq n}p_{ij}$. Our sufficient condition for uniform consistency depends on the edge probabilities only through $p_{\min}$.

\begin{theorem}[Informal version of \Cref{thm:ger}]\label{thm:informal}
Suppose the dynamic range of $\bm u^*$ is bounded and 
$np_{\min}=\omega((\log n)^2)$. Then, with probability tending to one, 
the MLE exists uniquely and
\[
\max_{1\leq i< j\leq n}|(\widehat u_i-\widehat u_j)-(u_i^*-u_j^*)|\lesssim
\sqrt{\frac{\log n}{np_{\min}}}+\frac{(\log n)^2}{np_{\min}}
=o(1).
\]
\end{theorem}

Thus uniform consistency holds as soon as the minimum comparison probability exceeds the Erd\H{o}s--R\'enyi connectivity threshold $\log n/n$ by a logarithmic factor. No upper bound on $p_{ij}$, block structure, degree-balance condition, or compatibility condition is required, so the design may be arbitrarily heterogeneous above the floor $p_{\min}$.

Our bound \eqref{bdd-intro} also yields an asymptotic normality result for any prescribed sequence of pairs under a suitable balancing condition. In particular, \Cref{thm:norm} informally states that, for any fixed sequence of pairs $(i_n, j_n)$, asymptotic normality holds whenever the linearized error term dominates the global remainder term by a suitable margin. A sufficient condition is
\begin{align*}
\lambda_2^2(\bm L_{\G_n})\rr_{\G_n}(i_n\leftrightarrow j_n)=\omega((\log n)^4).
\end{align*}
Under this condition,
\begin{align*}
\frac{(\widehat{u}_{i_n}-\widehat{u}_{j_n})-(u_{i_n}^*-u_{j_n}^*)}
{\sqrt{\rr_{\G_n^*}(i_n\leftrightarrow j_n)}}
\Rightarrow \mathcal N(0,1),\qquad\text{as $n\to\infty$}, 
\end{align*}
where $\G_n^*$ is a weighted graph based on $\G_n$ whose edge weights depend on $\bm u^*$ and are chosen so that the corresponding effective resistance gives the appropriate variance normalization. 

Our approach combines the electrical-network interpretation of pairwise MLE errors in \cite{chen2023ranking} with a new perturbation strategy that avoids auxiliary compatibility conditions. \citet{chen2023ranking} applies a preconditioned gradient descent to the negative log-likelihood initialized at $\bm u^*$ and controls the accumulated error over the entire sequence of iterates. Closing the resulting induction requires a compatibility condition \cite[Condition~(8)]{chen2023ranking} involving certain abstract graph-dependent quantities, which restricts the applicability of the method. By contrast, we take only one preconditioned gradient step from $\bm u^*$ and directly analyze the perturbation around the resulting point. This strategy avoids the compatibility condition but requires pathwise control of the Hessian of the log-likelihood. To address this, we construct a surrogate likelihood and apply a localization argument that is similar in spirit to the primal-dual witness method in high-dimensional statistics \cite{wainwright2009sharp}. This forms the first key new ingredient of our analysis.

The resulting perturbation analysis requires controlling a nonlinear remainder through the random graph-theoretic quantity $Q_{ij}$ in \eqref{pang1}. Our second technical contribution is a uniform upper bound for this quantity in terms of the unnormalized spectral gap. To establish this bound, we connect $Q_{ij}$ to the $\ell_1$-deviation of the voltage profile generated by a unit current injected at $i$ and withdrawn at $j$, and control this deviation using a level-set argument. Conceptually, this argument may be viewed as a deterministic analogue of graph-based chaining methods for controlling MLE errors \cite{han2023general,han2023unified,dong2026statistical}; see \Cref{sec:l1} for further discussion of this connection.

Our analysis unifies the electrical-network and graph-chaining approaches and illustrates how local and global graph properties jointly determine the pairwise MLE error. Once the MLE error bound is established, asymptotic normality follows under the additional balancing condition \eqref{eq:an-pair-condition} by applying the Lindeberg central limit theorem to the linearized term.

\paragraph{Organization and notation.}
The remainder of the article is organized as follows. \Cref{sec:ps} formulates the problem of uniform consistency for the MLE and introduces the necessary notation. \Cref{sec:en} provides a self-contained overview of the electrical-network interpretation of the comparison graph, develops the main intuition underlying our analysis, and establishes a technical estimate used in subsequent sections. \Cref{sec:unif} presents our main asymptotic results for the MLE under general comparison graphs and applies them to several random graph designs. The technical proofs are deferred to \Cref{sec:proof}. \Cref{sec:num} conducts numerical simulations to verify our theoretical findings. \Cref{sec:con} concludes the article and discusses a few future directions. Throughout the article, we use Bachmann--Landau asymptotic notation. In particular, for two positive sequences $a_n$ and $b_n$, we write $a_n=o(b_n)$ if $a_n/b_n\to0$ as $n\to\infty$, and $a_n=\omega(b_n)$ if $b_n=o(a_n)$. 

\section{Problem setup}\label{sec:ps}

A pairwise comparison design involving $n$ objects can be represented by an undirected weighted graph $\G=([n], \M)$, where $[n]\coloneqq\{1, \ldots, n\}$ and $\M\coloneqq\{m_{ij}\}_{1\leq i<j\leq n}$, with $m_{ij}\geq 0$ denoting the number of comparisons between objects $i$ and $j$. We extend the weights symmetrically by setting $m_{ji}=m_{ij}$. Objects $i$ and $j$ are said to be compared if $m_{ij}>0$ and denoted by $i\sim j$. 

Let $\bm e_i\in\R^n$ denote the $i$th canonical basis vector, and define $\bm x_{ij}\coloneqq\bm e_i-\bm e_j$. Fixing an orientation $i\to j$ if $i<j$, we define the \textit{full} incidence matrix of $\G$ as 
\begin{align}
\bm Q\coloneqq(\bm x_{ij})_{1\leq i<j\leq n}\in\R^{n\times {n\choose 2}}.\label{Q}
\end{align}
The range of $\diag(\M)\bm Q^\top$ consists of the weighted gradient flows on $\G$, with the flow on $\{i, j\}$ set to zero whenever $i\nsim j$; see \Cref{sec:sger} for further interpretations.

In the BT model, the probability of object $i$ defeating object $j$, denoted by $i\succ j$, in a single comparison is given by
\begin{align*}
\mathbb{P}(i\succ j) = \sigma(\bm x_{ij}^\top\bm u) = \sigma(u_i-u_j),\qquad\sigma(u)\coloneqq(1+e^{-u})^{-1}.
\end{align*}
This probability depends on $\bm u$ only through the utility difference $u_i-u_j$. For identifiability, we impose the centering constraint
$\bm1^\top\bm u=0$. The true utility vector is denoted by $\bm u^*=(u_1^*,\ldots,u_n^*)^\top\in\R^n$, where
$\bm1^\top\bm u^*=0$, and its dynamic range is defined as $\kappa\coloneqq\max_{1\leq i\leq n}u_i^*-\min_{1\leq i\leq n}u_i^*$.

Conditional on the comparison design $\G$, we assume that all individual comparison outcomes are independent. For every $i\neq j$, let $Y_{ij}$ denote the number of the $m_{ij}$ comparisons in which $i$ defeats $j$. Thus,
\begin{align*}
Y_{ij}\sim\operatorname{Binomial}\left(m_{ij},\sigma(\bm x_{ij}^\top\bm u^*)\right), \qquad 1\leq i<j\leq n,
\end{align*}
independently across edges. The conditional log-likelihood of
$\bm u$ is
\begin{equation}\label{mle}
\begin{aligned}
l(\bm u)&=\sum_{1\leq i<j\leq n}
\left[Y_{ij}\log\sigma(\bm x_{ij}^\top\bm u)+(m_{ij}-Y_{ij})\log\bigl(1-\sigma(\bm x_{ij}^\top\bm u)\bigr)
\right]\\
&=\sum_{1\leq i<j\leq n}\left[Y_{ij}\bm x_{ij}^\top\bm u-m_{ij}\psi(\bm x_{ij}^\top\bm u)\right],\qquad\psi(u)\coloneqq\log(1+e^u).
\end{aligned}
\end{equation}
The MLE of $\bm u^*$ is defined by
\begin{align*}
\uu\coloneqq\argmax_{\bm u:\,\bm1^\top\bm u=0}\; l(\bm u).
\end{align*}
\citet{MR0097876} showed that the MLE uniquely exists and is finite if and only if the directed graph induced by the comparison outcomes is strongly connected. This is equivalent to saying that any Markov chain whose transition graph coincides with this directed comparison graph is irreducible.

Uniform consistency of the MLE concerns conditions under which all components of $\bm u^*$ can be consistently estimated simultaneously by $\uu$. Because the BT model identifies utilities only up to a common additive constant, an intrinsic measure of estimation error is
\begin{align}
\Delta_n\coloneqq\max_{1\leq i<j\leq n}\left|(\widehat u_i-\widehat u_j)-(u_i^*-u_j^*)\right|.\label{maxerror}
\end{align}
Under the centering constraints $\bm1^\top\uu=\bm1^\top\bm u^*=0$, this quantity satisfies
\begin{align*}
\|\uu-\bm u^*\|_\infty\leq \Delta_n\leq 2\|\uu-\bm u^*\|_\infty.
\end{align*}
Thus, uniform consistency may equivalently be formulated as $\Delta_n=o_{\P}(1)$ or $\|\uu-\bm u^*\|_\infty=o_{\P}(1)$. Throughout this article, we use $\Delta_n$. 

Roughly speaking, $\Delta_n$ depends on how efficiently information propagates within the comparison graph. Recent work has related this error to several graph-theoretic properties, including the spectral gap \cite{pmlr-v180-li22g}, effective resistance \cite{chen2023ranking}, and unnormalized Cheeger constant and related graph diameters \cite{han2023general}. Although these quantities are closely related to each other under special settings, they can behave quite differently for general comparison graphs. We next review them from an electrical-network perspective and establish a key technical lemma for later use. 

\section{Electrical networks}\label{sec:en}

\subsection{Spectral gap and effective resistance}\label{sec:sger}

Consider a network $\G$ on $[n]$ with nonnegative weights $\mathcal C=\{c_{ij}\}_{1\leq i<j\leq n}$ ($c_{ji}=c_{ij}$ for $i<j$). Here we interpret $c_{ij}$ as the conductance (or the reciprocal resistance) between $i$ and $j$, and write $\bm C = \diag(\mathcal C)$. The corresponding incidence matrix $\bm Q$ is the same as defined in \eqref{Q}. In the following discussion, we assume $\G$ is connected. 

A vector $\bm v = (v_1, \ldots, v_n)^\top\in\R^n$ can be viewed as a voltage profile on $[n]$. For each $1\leq i\neq j\leq n$, Ohm's law states that the current flowing from $i$ to $j$ is 
\begin{align}
I_{i\to j}(\bm v)=c_{ij}(v_i-v_j), \qquad  \bm I(\bm v) = (I_{i\to j}(\bm v))_{1\leq i<j\leq n}\in\R^{{n\choose 2}}.\label{induced}
\end{align}
Hence, the net current flowing out of vertex $i$ through the edges of $\G$ (or the divergence of $\bm I(\bm v)$ at $i$) is $\sum_{j} I_{i\to j}(\bm v)$. In matrix form, the vector of net outgoing currents at each vertex is $\bm L_{\G}\bm v$, where $\bm L_{\G}\coloneqq\bm Q\bm C\bm Q^\top\in\R^{n\times n}$ is the unnormalized graph Laplacian of $\G$. When $\G$ is connected, $\ker(\bm L_{\G}) = \mathrm{span}{\{\bm 1\}}$. This reflects the fact that adding a constant to all voltages does not change the induced current flow. For identifiability, we assume $\bm 1^\top\bm v=0$. The second-smallest eigenvalue of $\bm L_{\G}$ is denoted by $\lambda_2(\bm L_{\G})$; it measures the algebraic connectivity of the graph and is referred to as the \textit{spectral gap} of $\bm L_{\G}$.

For any $i\neq j\in[n]$, inject one unit of current into the network at $i$ and extract one unit at $j$ while leaving the remaining vertices unaffected. The corresponding vector of external current injections is $\bm x_{ij}$. By Kirchhoff's current law, the net current flowing out of each vertex through the network is equal to the externally injected current. Therefore, the resulting voltage profile satisfies 
\begin{align}
\bm L_{\G}\bm v=\bm x_{ij}. \label{itoj0}
\end{align}
Under the normalization $\bm 1^\top\bm v=0$, this equation has the unique solution 
\begin{align}
\bm v^{i\to j}=\bm L_{\G}^{\dagger}\bm x_{ij}. \label{itoj}
\end{align}
The voltage difference between $i$ and $j$ is 
\begin{align}
\rr_{\G}(i\leftrightarrow j)\coloneqq v^{i\to j}_i-v^{i\to j}_j=\bm x_{ij}^\top\bm v^{i\to j}=\bm x_{ij}^\top\bm L_{\G}^{\dagger}\bm x_{ij},
\end{align}
which is a quadratic form of $\bm L_{\G}^{\dagger}$ symmetric in $i$ and $j$. Since a unit current is injected at $i$ and extracted at $j$, this voltage difference is precisely the \textit{effective resistance} between $i$ and $j$, i.e., $\rr_{\G}(i\leftrightarrow j)$ is the resistance of a single resistor that is electrically equivalent to the entire network when viewed from terminals $i$ and $j$. The \textit{maximum principle} states that
\begin{align}
v^{i\to j}_k-v^{i\to j}_\ell\leq v^{i\to j}_i-v^{i\to j}_j = \rr_{\G}(i\leftrightarrow j), \qquad \forall k, \ell\in [n], \label{mp}
\end{align}
since $\bm v^{i\to j}$ is harmonic at every vertex other than the source $i$ and the sink $j$.
The maximum effective resistance of the network $(\G,\mathcal C)$ is denoted by
\begin{align}
\rr_{\G}\coloneqq \max_{1\leq i<j\leq n} \rr_{\G}(i\leftrightarrow j).
\end{align}

Both $1/\lambda_2(\bm L_{\G})$ and $\rr_{\G}$ are defined via quadratic forms of $\bm L_{\G}^\dagger$ (Green's function). This naturally gives the following monotonicity principle, which is well known, and we provide the proof for completeness. 

\begin{lemma}[Rayleigh's monotonicity principle]\label{lm:rp}
Consider two networks $\G$ and $\G'$ on $[n]$ with nonnegative conductances $\mathcal C=\{c_{ij}\}_{1\leq i<j\leq n}$ and $\mathcal C'=\{c'_{ij}\}_{1\leq i<j\leq n}$, respectively. Suppose that $c_{ij}\geq \alpha c'_{ij}$ for some $\alpha > 0$ and all $1\leq i<j\leq n$ and that $\G'$ is connected. Then
\begin{align*} \frac{1}{\lambda_2(\bm L_{\G})}\leq\frac{1}{\alpha\lambda_2(\bm L_{\G'})},\qquad \rr_{\G}(i\leftrightarrow j)\leq \frac{1}{\alpha}\rr_{\G'}(i\leftrightarrow j), \qquad \rr_{\G}\leq \frac{1}{\alpha}\rr_{\G'}. 
\end{align*}
\end{lemma}
\begin{proof}
Note for every $\bm z=(z_1, \ldots, z_n)^\top\in\R^n$, $\bm z^\top\bm L_\G\bm z=\sum_{1\leq i<j\leq n}c_{ij}(z_i-z_j)^2$. As a result, $\bm L_{\G}\succeq\alpha\bm L_{\G'}\succeq 0$. Since $\G'$ is connected, $\ker(\bm L_{\G})=\ker(\alpha\bm L_{\G'}) = \mathrm{span}\{\bm 1\}$. Restricting to $\mathrm{span}\{\bm 1\}^\perp$, since $x\mapsto 1/x$ is operator decreasing on $\R_+$, $\bm L^\dagger_{\G}\preceq (\alpha\bm L_{\G'})^\dagger = \bm L^\dagger_{\G'}/\alpha$. The proof is finished by using the definitions of $\lambda_2(\bm L_\G)$ and $\rr_\G(i\leftrightarrow j)$. 
\end{proof}
Despite their shared monotonicity with respect to the edge weights, the reciprocal spectral gap and effective resistance can exhibit rather different behavior. Indeed, by the variational characterization of $\lambda_2(\bm L_{\G})$ and a standard spectral estimate, 
\begin{align}
\frac{1}{n\lambda_2(\bm L_{\G})}\leq \rr_{\G}\leq \frac{2}{\lambda_2(\bm L_{\G})}, \label{transfer}
\end{align}
and both scalings are tight for certain graphs up to constant factors; see, e.g., \cite[Theorem 4.2]{chandra1989electrical}. Thus, while the spectral gap captures the global connectivity of the network, effective resistance can encode finer geometric information; see \Cref{prop:er}. 

Effective resistance admits several useful probabilistic interpretations. For instance, for an edge $\{i,j\}$, the quantity $c_{ij}\rr_{\G}(i\leftrightarrow j)$ is precisely its leverage score with respect to the weighted incidence matrix $\bm C^{1/2}\bm Q^\top$. We also introduce the vertex leverage score by aggregating the leverage scores of the edges incident to a vertex:
\begin{align}
\ell_{\G}(i)\coloneqq\sum_{j\neq i} c_{ij}\rr_{\G}(i\leftrightarrow j).\label{lev}
\end{align}
This quantity will arise naturally in our analysis and is closely related to the resistance curvature introduced by \citet{devriendt2022discrete}, which is defined by $p_i=1-\ell_{\G}(i)/2$. Let $\bm\ell_{\G}\coloneqq(\ell_{\G}(1),\ldots,\ell_{\G}(n))^\top$. For a connected graph $\G$, one has
\begin{align}
\bm\ell_{\G}=2\left(1-\frac{1}{n}\right)\bm 1-\bm L_{\G}\diag\mathrm{vec}(\bm L_{\G}^{\dagger}),\label{lev1}
\end{align}
where $\diag\mathrm{vec}(\bm L_{\G}^{\dagger})$ denotes the vector formed by the diagonal entries of $\bm L_{\G}^{\dagger}$; see, e.g., \cite[Equation~(19)]{devriendt2022discrete}. 
For additional probabilistic interpretations of effective resistance and related graph quantities, see \cite{lyons2016probability}.

\subsection{$\ell_1$-deviation of voltage profiles}\label{sec:l1}
Both the reciprocal spectral gap and effective resistance are associated with the $\ell_2$-geometry weighted by the graph Laplacian pseudoinverse. In \Cref{sec:proof}, we also need an $\ell_1$-geometric quantity associated with a voltage profile, namely, its $\ell_1$-total deviation about a median:
\begin{align}
\sp(\bm v)\coloneqq\inf_{c\in\R}\|\bm v-c\bm 1\|_1=\|\bm v-\med(\bm v)\bm 1\|_1,\label{mysv}
\end{align}
where $\med(\bm v)$ denotes any median of the coordinates of $\bm v$. This quantity is closely related to the unnormalized
Cheeger constant of the underlying weighted graph, which is used in \cite{han2023general} and recalled below. 

\begin{definition}[Unnormalized Cheeger constant]\label{def:cheeger}
Let $\G$ be a connected weighted graph with nonnegative edge weights $\mathcal C=\{c_{ij}\}_{1\leq i<j\leq n}$. Its unnormalized Cheeger
constant is
\begin{align}
h_\G \coloneqq \min_{\emptyset\neq U\subset [n]}\frac{c(U, U^\complement)}{\min\{|U|, |U^\complement|\}}, \qquad c(U, U^\complement)\coloneqq \sum_{i\in U, j\in U^\complement}c_{ij}. \label{cheeger}
\end{align}
\end{definition}

Recall $\bm I(\bm v) = (I_{i\to j}(\bm v))_{1\leq i<j\leq n}$ defined in \eqref{induced}. The unnormalized Cheeger constant admits the following variational characterization~\cite[Theorem~1]{chung1998isoperimetric}:
\begin{align}
h_{\G}=\inf_{\bm v\in\R^n: \bm v\notin\text{span}\{\bm 1\}}\frac{\|\bm I(\bm v)\|_1}{\sp(\bm v)}.\label{cv98}
\end{align}
Consequently, for any $\bm v\in\R^n$, 
\begin{align}
\sp(\bm v)\leq\frac{\|\bm I(\bm v)\|_1}{h_\G}.\label{crude}
\end{align}
The proof of \eqref{cv98} relies on a level-set decomposition which is relevant to the proof of \Cref{lem:voltage-spread}. For completeness, we sketch the proof here.

\begin{proof}[Proof of \Cref{cv98}]
Fix $\bm v\notin\text{span}\{\bm 1\}$ and take $\bar{v}=\med(\bm v)$. For $t>0$, define the following upper and lower level sets $U_t^+=\{i: v_i>\bar{v}+t\}$ and $U_t^-=\{i: v_i<\bar{v}-t\}$. Using the integral representations of $\|\bm I(\bm v)\|_1$ and $\sp(\bm v)$ in terms of $U^{\pm}_t$, we have 
\begin{align*}
\|\bm I(\bm v)\|_1&=\int_0^\infty c(U_t^+, (U_t^+)^\complement) \d t + \int_0^\infty c(U_t^-, (U_t^-)^\complement) \d t,\\
\sp(\bm v)&=\int_0^\infty |U_t^+|\d t+\int_0^\infty |U_t^-|\d t\\
 &= \int_0^\infty \min\{|U_t^{+}|, |(U_t^{+})^\complement|\}\d t+\int_0^\infty \min\{|U_t^{-}|, |(U_t^{-})^\complement|\}\d t,
\end{align*}
where the last step follows from the definition of the median. Taking the ratio between $\|\bm I(\bm v)\|_1$ and $\sp(\bm v)$ and invoking the definition of $h_\G$ in \eqref{cheeger} yields $h_{\G}\leq\|\bm I(\bm v)\|_1/\sp(\bm v)$. Taking the infimum over $\bm v$ yields the lower bound in \eqref{cv98}. For the upper bound, take $\bm v$ as an indicator function on the partition $U$ that attains the equality in the minimum in \eqref{cheeger}. 
\end{proof}

This level-set decomposition is closely related to the graph-based chaining ideas used in the study of the BT model \cite{han2020asymptotic, han2023general}. These arguments track nested level sets of the MLE error vector $\uu-\bm u^*$, combining concentration bounds for aggregate utility fluctuations with graph expansion to show that the upper and lower error level sets grow until they intersect. The deterministic part of this chaining procedure is analogous to the level-set decomposition in the proof of \eqref{cv98}.

A quantity that arises in our analysis for the entrywise MLE error between objects $i$ and $j$ is $\sp(\bm v^{i\to j})$, where $\bm v^{i\to j}$ is defined in \eqref{itoj}. A direct application of \eqref{crude} (e.g., combined with a Cauchy--Schwarz bound on $\|\bm I(\bm v^{i\to j})\|_1$ using effective resistance) may be too crude because $\bm v^{i\to j}$ need not resemble a minimizer in the variational characterization \eqref{cv98}. We instead combine a truncation argument with the level-set decomposition underlying \eqref{cv98} to obtain a sharper bound.

\begin{lemma}[$\ell_1$-deviation of $\bm v^{i\to j}$]
\label{lem:voltage-spread}
Let $\G$ be a weighted graph on $[n]$ with weights $\mathcal C=\{c_{ij}\}_{1\leq i<j\leq n}$ and assume $\G$ is connected. 
Let $\bm v^{i\to j}$ be the voltage profile defined in \eqref{itoj}. Then
\begin{align}
\sp(\bm v^{i\to j})\leq\frac{4(1+\log n)}{\lambda_2(\bm L_{\G})}. 
\label{eq:voltage-spread}
\end{align}
\end{lemma}

\begin{proof}
Write $\bm v=\bm v^{i\to j}$ and $\bar{v} = \med(\bm v)$. For $t>0$, define the truncated positive and negative parts
\begin{align*}
\bm f^t\coloneqq\min\{(\bm v-\bar{v}\bm 1)_+, t\bm 1\},\qquad \bm g^t\coloneqq\min\{(\bar{v}\bm 1-\bm v)_+, t\bm 1\},
\end{align*}
where the minimum is taken componentwise. 

We first analyze $\bm f^t$ to bound $\|(\bm v-\bar{v}\bm 1)_+\|_1$. Let $\bm f^{t, c} = \bm f^t - \bm 1\bm 1^\top\bm f^t/n$ denote the centered $\bm f^t$.  By the definition of the median, $|\supp(\bm f^t)|\leq n/2$. By the bias-variance decomposition and applying the Cauchy--Schwarz inequality restricted to the components in $\supp(\bm f^t)$, 
\begin{align}
\|\bm f^{t, c}\|_2^2 = \|\bm f^t\|_2^2 - \frac{1}{n}(\bm 1^\top\bm f^t)^2\geq \|\bm f^t\|_2^2-\frac{1}{2}\|\bm f^{t}\|_2^2=\frac{1}{2}\|\bm f^{t}\|_2^2.\label{center} 
\end{align}

For any $r\leq n/2$, let $t_r$ denote the $r$th largest component of $(\bm v-\bar{v}\bm 1)_+$. We now derive an upper bound on $t_r$ for each positive $t_r$. By definition, $\bm f^{t_r}$ has at least $r$ components equal to $t_r$. Using \eqref{center} and the definition of $\lambda_2(\bm L_\G)$,  
\begin{align}
(\bm f^{t_r,c})^\top\bm L_{\G}\bm f^{t_r,c}\geq \lambda_2(\bm L_\G)\|\bm f^{t_r, c}\|_2^2\geq \frac{\lambda_2(\bm L_\G)}{2}\|\bm f^{t_r}\|_2^2\geq \frac{\lambda_2(\bm L_\G)}{2}rt_r^2. \label{low}
\end{align} 
Meanwhile, for any $k, \ell\in [n]$, $c_{k\ell}(f^{t_r}_k-f^{t_r}_\ell)^2\leq c_{k\ell}(v_k-v_\ell)(f^{t_r}_k-f^{t_r}_\ell)$. Summing this over $1\leq k<\ell\leq n$ gives
\begin{align}
(\bm f^{t_r,c})^\top\bm L_{\G}\bm f^{t_r,c}=(\bm f^{t_r})^\top\bm L_{\G}\bm f^{t_r}\leq\bm v^\top\bm L_{\G}\bm f^{t_r} =
(\bm L_{\G}\bm v)^\top\bm f^{t_r}\stackrel{\eqref{itoj}}{=}\bm x_{ij}^\top\bm f^{t_r}=f^{t_r}_i-f^{t_r}_j\leq t_r.\label{upp}
\end{align}
Combining \eqref{upp} and \eqref{low} yields
\begin{align*}
t_r\leq\frac{2}{\lambda_2(\bm L_\G)r}.
\end{align*}
Summing over $r$ gives 
\begin{align*}
\|(\bm v-\bar{v}\bm 1)_+\|_1 = \sum_{r=1}^{\lfloor n/2\rfloor} t_r\leq \frac{2}{\lambda_2(\bm L_\G)}\sum_{r=1}^{\lfloor n/2\rfloor}\frac{1}{r}\leq\frac{2(\log (n/2) + 1)}{\lambda_2(\bm L_\G)}.
\end{align*}
Applying the same argument to $\bm g^t$ yields
\begin{align*}
\|(\bar{v}\bm 1-\bm v)_+\|_1\leq\frac{2(\log (n/2) + 1)}{\lambda_2(\bm L_\G)}. 
\end{align*}
The proof is completed by noting $\sp(\bm v^{i\to j})=\|(\bm v-\bar{v}\bm 1)_+\|_1+\|(\bar{v}\bm 1-\bm v)_+\|_1$. 
\end{proof}

\section{Asymptotic theory}\label{sec:unif}

\subsection{Main results}\label{sec:unif1}

We first state our uniform consistency result along a sequence of deterministic connected comparison graphs $\G_n=([n],\M_n)$, where $\M_n=\{m_{ij}\}_{1\leq i<j\leq n}$ records the numbers of comparisons. Let $h_{\G_n}$ denote the unnormalized Cheeger constant of $\G_n$ (see \Cref{def:cheeger}) and $\bm L_{\G_n}$ its unnormalized graph Laplacian. We write $\lambda_2(\bm L_{\G_n})$ for the corresponding spectral gap, $\rr_{\G_n}(i\leftrightarrow j)$ for the effective resistance between $i$ and $j$, and $\rr_{\G_n}$ for the maximum effective resistance over all pairs of vertices. 

We also consider a weighted version of $\G_n$, denoted by $\G_n^*=([n],\M_n^*)$, with $\M_n^*=\{m_{ij}\sigma'(\bm x_{ij}^\top\bm u^*)\}_{1\leq i<j\leq n}$, where $\sigma(x)$ is the logistic function. The corresponding graph quantities are defined analogously for $\G_n^*$.

The dynamic range of $\bm u^*$ on $\G_n$ is denoted by $\kappa_n$. For convenience, we impose the following uniform boundedness assumption on $\kappa_n$:

\begin{assumption}\label{ass:bdd}
The dynamic range of the objects under comparison is uniformly bounded along the graph sequence:
$\kappa\coloneqq\sup_n \kappa_n<\infty$.
\end{assumption}

Assumption~\ref{ass:bdd} can be relaxed to allow $\kappa_n$ to grow slowly with $n$, but we do not pursue this extension to keep the statements simple. Under Assumption~\ref{ass:bdd}, the edge weights of $\G_n$ and $\G_n^*$ are equivalent up to a constant depending only on $\kappa$: for all $i\neq j$ with $m_{ij}>0$, 
\begin{align}
\frac{1}{4e^{\kappa}}\leq\frac{m_{ij}\sigma'(\bm x_{ij}^\top\bm u^*)}{m_{ij}}\leq \frac{1}{4}.\label{yh}
\end{align}
In particular, their effective resistances and spectral gaps are equivalent by \Cref{lm:rp}. For convenience, we formulate our result in terms of $\G_n^*$, which also appears in \Cref{thm:norm}. When verifying these conditions later, we use $\G_n$ instead, as illustrated in \Cref{sec:exam}.

\begin{theorem}[Uniform consistency]\label{thm:main}
Suppose 
\begin{align}
\lambda_2(\bm L_{\G^*_n})=\omega((\log n)^2).\label{unib}
\end{align}
Under Assumption~\ref{ass:bdd}, for all large $n$, with probability at least $1-n^{-2}$, the MLE uniquely exists, and simultaneously for all distinct $i, j\in [n]$,  
\begin{align}\label{main:bdd}
|(\widehat{u}_i-\widehat{u}_j)-(u_i^*-u_j^*)|\lesssim \sqrt{e^{\kappa}\rr_{\G^*_n}(i\leftrightarrow j)\log n} + \frac{e^{\kappa} (\log n)^2}{\lambda_2(\bm L_{\G^*_n})},\qquad i, j\in [n].
\end{align}
Here and throughout, the implicit constant is absolute and independent of both $n$ and $\kappa$. Consequently, by \eqref{transfer}, the uniform error $\Delta_n$ in \eqref{maxerror} satisfies
\begin{equation}\label{main:glob}
\begin{aligned}
\Delta_n \lesssim \sqrt{e^{\kappa}\rr_{\G^*_n}\log n} + \frac{e^{\kappa} (\log n)^2}{\lambda_2(\bm L_{\G^*_n})}\leq\sqrt{\frac{2e^{\kappa}\log n}{\lambda_2(\bm L_{\G^*_n})}} + \frac{e^{\kappa} (\log n)^2}{\lambda_2(\bm L_{\G^*_n})}=o(1). 
\end{aligned}
\end{equation}
\end{theorem}

The proof of the full theorem is given in \Cref{sec:main}. The bound in \eqref{main:bdd} consists of a pair-specific term governed by effective resistance and a global term governed by the spectral gap of $\G^*_n$. Unlike \cite{pmlr-v180-li22g}, our bound does not involve a degree-imbalance factor, such as the ratio of the maximum to the minimum degree, which can substantially restrict the allowable graph heterogeneity. The analyses in \cite{chen2023ranking,yang2024top} yield a similar abstract decomposition involving effective resistance but require additional self-consistency conditions, whose verification can be nontrivial in general. Interestingly, \Cref{lm:tech}, which is a key technical estimate in our proof, also provides a concrete way to verify this condition. For instance, under Assumption~\ref{ass:bdd}, for \cite[Condition~(8)]{chen2023ranking} to hold, it suffices to require
\begin{align}
\min_{1\leq i<j\leq n: m_{ij}>0}\frac{\lambda_2^2(\bm L_{\G^*_n})\rr_{\G^*_n}(i\leftrightarrow j)}{(\log n)^3}\geq C(\kappa)\label{chenc}
\end{align}
for some large constant $C(\kappa)>0$ depending only on $\kappa$. 

Our analysis also yields an alternative estimate that replaces the global remainder in \eqref{main:bdd} by $e^\kappa n\rr_{\G_n^*}(i\leftrightarrow j)$; see Remark~\ref{rem:classical}. For fixed $\kappa$, this bound asymptotically vanishes if $\rr_{\G_n^*}(i\leftrightarrow j)=o(n^{-1})$. This resistance condition requires each object in the pair to participate in $\omega(n)$ comparisons. The refinement captures the effect of many direct comparisons between $i$ and $j$ and reflects the classical setting of fixed $n$ and increasing comparison counts. However, the
comparison requirement can be restrictive when $n$ grows. Therefore, we retain the spectral-gap term in the main statement to keep the bound and its interpretation simple.

Although effective resistance depends on the entire comparison graph, it is closely related to local degree information when the graph is well-connected (e.g., $\lambda_2(\bm L_{\G_n})=\omega(1)$). To make this connection explicit, consider the setting in which every pair is compared at most once, i.e., $\G_n$ is a simple graph. 
\begin{proposition}\label{prop:er}
Suppose $m_{ij}\leq 1$ for all $1\leq i<j\leq n$. Then, for every distinct $i, j\in[n]$,
\begin{align}
\frac{1}{d_i}+\frac{1}{d_j}-\frac{2m_{ij}}{d_id_j}\leq\rr_{\G_n}(i\leftrightarrow j)\leq\left[\left(1+\frac{1}{\lambda_2(\bm L_{\G_n})}\right)\left(\frac{1}{d_i}+\frac{1}{d_j}\right)-\frac{2m_{ij}}{d_id_j}\right], \label{locr}
\end{align} 
where $d_i$ denotes the degree of $i$ in $\G_n$. 
\end{proposition}
Similar reciprocal-degree approximations were established by \citet{von2014hitting} using the normalized spectral gap and minimum degree.
The proof of \Cref{prop:er} uses an approximation argument based on the electrical-network representation and is deferred to \Cref{sec:prop}. Combining \eqref{locr} with \eqref{yh} and \Cref{lm:rp} yields analogous bounds for $\rr_{\G_n^*}(i\leftrightarrow j)$, up to constant factors depending only on $\kappa$.

The decomposition in \eqref{main:bdd} also leads to an asymptotic normality result. The first pair-specific term corresponds to the linearized fluctuation of the MLE and has variance governed by the effective resistance, whereas the second term is a global nonlinear remainder that serves as a bias term. When the pair-specific fluctuation dominates the global remainder by a suitable margin, the latter becomes asymptotically negligible after normalization by the square root of the corresponding effective resistance. The asymptotic distribution is then determined by the linearized term, to which a standard central limit theorem can be applied.

\begin{theorem}[Asymptotic normality]\label{thm:norm}
Suppose Assumption~\ref{ass:bdd} holds. Let $(i_n,j_n)$ be any prespecified sequence of distinct vertex pairs, with $i_n,j_n\in[n]$. If
\begin{align}
\lambda_2^2(\bm L_{\G_n^*})\rr_{\G_n^*}(i_n\leftrightarrow j_n)=\omega((\log n)^4),
\label{eq:an-pair-condition}
\end{align}
then the MLE exists uniquely with probability tending to one, and
\begin{align}
\frac{(\widehat u_{i_n}-\widehat u_{j_n})-(u_{i_n}^*-u_{j_n}^*)}
{\sqrt{\rr_{\G_n^*}(i_n\leftrightarrow j_n)}}
\Rightarrow \mathcal N(0,1),
\qquad \text{as } n\to\infty.
\label{eq:an-oracle-clt}
\end{align}
Moreover, the same conclusion holds if $\rr_{\G_n^*}(i_n\leftrightarrow j_n)$ in \eqref{eq:an-oracle-clt} is replaced by its plug-in estimator based on the MLE.
\end{theorem}

\begin{remark}\label{rem:1}
By \eqref{transfer}, condition \eqref{eq:an-pair-condition} implies $\lambda_2(\bm L_{\G_n^*})=\omega((\log n)^4)$, and hence, in particular, the  sufficient condition \eqref{unib} for the uniform consistency result.
\end{remark}

The proof is given in \Cref{sec:an}. In contrast to existing asymptotic normality results for the BT model under random comparison designs \cite{han2020asymptotic,gao2023uncertainty,han2024statistical}, \Cref{thm:norm} applies to a fixed comparison design. Moreover, condition \eqref{eq:an-pair-condition} is expressed directly in terms of the spectral gap and the effective resistance and is therefore relatively easy to verify, in contrast to the more abstract conditions in \citet{han2023unified}. An important feature of our result is that it applies to any prespecified sequence of vertex pairs rather than requiring a condition to hold uniformly over all pairs. Consequently, \eqref{eq:an-pair-condition} may hold for some pair sequences but fail for others. This pair-specific formulation captures local heterogeneity in the comparison graph and provides a finer characterization of when asymptotic normality holds.

\subsection{Examples}\label{sec:exam}

We now apply the results of \Cref{sec:unif1} to several specific comparison graph designs. We first consider a generalized Erd\H{o}s--R\'enyi model in which edges are formed independently with potentially different probabilities. This model includes the standard Erd\H{o}s--R\'enyi model and the SBM as special cases while allowing for more general forms of heterogeneity. For ease of exposition, we focus on the setting in which each pair of objects is compared at most once. The case of repeated comparisons can be analyzed similarly, at the cost of heavier notation.

\begin{theorem}[Generalized Erd\H{o}s--R\'enyi model]\label{thm:ger}
For each $n$, let $\G_n=([n], \M_n)$ with edges $\M_n=\{m_{ij}^{(n)}\}_{1\leq i<j\leq n}$ generated by
\begin{align*}
 m_{ij}^{(n)}=A_{ij}^{(n)},\qquad A_{ij}^{(n)}\sim\operatorname{Bernoulli}\bigl(p_{ij}^{(n)}\bigr),\qquad 1\le i<j\le n,
\end{align*}
where the random variables $\{A_{ij}^{(n)}\}_{1\leq i<j\leq n}$ are independent. Set the minimum comparison rate as $p^{(n)}_{\min}\coloneqq\min_{1\le i<j\le n}p_{ij}^{(n)}$. Let $\bar{d}^{(n)}_i = \sum_{j\neq i}p_{ij}^{(n)}$ denote the expected degree at vertex $i$, and set the maximum expected degree among all vertices as $\bar{d}^{(n)}_{\max} = \max_i\bar{d}^{(n)}_i$.  
\begin{enumerate}
\item [(i)] If $np_{\min}^{(n)} = \omega((\log n)^2)$, then $\lambda_2(\bm L_{\G_n})\gtrsim np_{\min}^{(n)} =\omega((\log n)^2)$ a.s. 
Under Assumption~\ref{ass:bdd}, the weighted graph sequence $\G_n^*$ thus satisfies the spectral-gap condition in \Cref{thm:main}. Consequently, the MLE exists uniquely with probability tending to one, and, simultaneously for all $i\neq j$,
\begin{align}
|(\widehat{u}_i-\widehat{u}_j)-(u_i^*-u_j^*)|\lesssim e^{\kappa}\sqrt{\frac{\log n}{\min\{\bar{d}_i^{(n)}, \bar{d}_j^{(n)}\}}} + \frac{e^{2\kappa} (\log n)^2}{np_{\min}^{(n)}}.\label{random-bdd}
\end{align}
\item [(ii)] If $(np_{\min}^{(n)})^2/\bar{d}^{(n)}_{\max}=\omega((\log n)^4)$, then $\min_{i< j}\lambda_2^2(\bm L_{\G_n})\rr_{\G_n}(i\leftrightarrow j)=\omega((\log n)^4)$ a.s. Under Assumption~\ref{ass:bdd}, for any fixed sequence $(i_n, j_n)$, the MLE exists uniquely with probability tending to one, and
\begin{align}\label{random-an}
\frac{(\widehat u_{i_n}-\widehat u_{j_n})-(u_{i_n}^*-u_{j_n}^*)}
{\sqrt{\rr_{\widehat{\G}_n}(i_n\leftrightarrow j_n)}}
\Rightarrow \mathcal N(0,1),
\qquad \text{as } n\to\infty,
\end{align}
where $\widehat{\G}_n$ is a weighted graph on $[n]$ with edge weights $\{m^{(n)}_{ij}\sigma'(\bm x_{ij}^\top\uu)\}_{1\leq i<j\leq n}$ and $\sigma$ is the logistic function.  
\end{enumerate}    
\end{theorem}
The proof is based on a coupling argument together with concentration estimates and is given in \Cref{sec:ger}. By Rayleigh monotonicity (\Cref{lm:rp}), the uniform consistency result remains valid when each observed edge is assigned an arbitrary positive integer number of independent comparisons.  In the homogeneous model in which $p_{ij}^{(n)}=p_{\min}^{(n)}=p_n$ for all $i\neq j$, the first term in \eqref{random-bdd} dominates whenever $np_n\gg(\log n)^3$. In this regime, \eqref{random-bdd} agrees with the usual optimal uniform rate \cite{chen2022partial}. More generally, \eqref{random-bdd} provides a baseline guarantee without imposing any upper bound on the heterogeneity of the edge probabilities. In particular, it includes the SBM settings in \cite{han2023general,han2023unified} as special cases. 

The asymptotic normality statement requires an additional balancing condition. Unlike alternative conditions imposed directly on the maximum comparison probability \cite{han2024statistical}, our condition is formulated in terms of the maximum expected degree. For example, setting $p_{ij}^{(n)}=1$ for a single chosen pair $\{i, j\}$ changes $\bar d_{\max}^{(n)}$ by at most one and therefore does not affect its asymptotic order. Thus, our condition is insensitive to a finite number of unusually large edge probabilities.

We next consider an imbalanced complete bipartite comparison graph to examine the distinct contributions of the pair-specific effective-resistance term and the global nonlinear remainder in \eqref{main:bdd}. For this design, the two terms can have different asymptotic orders. Such designs may arise in applications such as item response theory \cite{chen2025item}.

\begin{theorem}[Imbalanced complete bipartite graphs]\label{thm:bipartite}
For each $n\geq 4$, let $\mathcal V_{1,n}=\{1,2\}$ and $\mathcal V_{2,n}=\{3,\ldots,n\}$. Let $\G_n=L_n\mathcal K_{2,n-2}$ be the weighted complete bipartite graph with vertex partition $\mathcal V_{1,n}\sqcup\mathcal V_{2,n}$, where every edge has
weight $L_n$. Suppose $L_n=\omega((\log n)^2)$. Under Assumption~\ref{ass:bdd}, for all sufficiently large $n$,
with probability at least $1-n^{-2}$, the MLE $\uu$ exists
uniquely and, simultaneously for all distinct $i,j\in[n]$,
\begin{align}
\left|(\widehat u_i-\widehat u_j)-(u_i^*-u_j^*)\right|\lesssim
\begin{cases}
\displaystyle
e^\kappa
\sqrt{\frac{2\log n}{(n-2)L_n}}
+
\frac{e^{2\kappa}(\log n)^2}{L_n},
& i, j\in\mathcal V_{1,n},
\\[2.5ex]
\displaystyle
e^\kappa
\sqrt{\frac{\log n}{L_n}}
+
\frac{e^{2\kappa}(\log n)^2}{L_n},
& i,j\in\mathcal V_{2,n},
\\[2.5ex]
\displaystyle
e^\kappa
\sqrt{
\frac{(n-1)\log n}{2(n-2)L_n}
}
+
\frac{e^{2\kappa}(\log n)^2}{L_n},
& i\in\mathcal V_{1,n},\
j\in\mathcal V_{2,n},
\end{cases}
\label{bipartite:bdd}
\end{align}
where the implicit constant is absolute. In particular, all three bounds are $o(1)$. 

Moreover, in the special case $u_1^*= -u_2^*=\log 2$ and $u^*_i=0$ for $3\leq i\leq n$, the estimated utility difference between objects $1$ and $2$ admits the asymptotics 
\begin{align}
(\widehat u_1-\widehat u_2)-(u_1^*-u_2^*)=\frac{3}{\sqrt{(n-2)L_n}}Z_n+\frac{3}{4L_n}+o_{\P}(L_n^{-1}), \label{bipartite:expansion}
\end{align}
where $Z_n\Rightarrow\mathcal N(0,1)$ as $n\to\infty$.
\end{theorem}

\begin{remark}
The uniform consistency result can also be obtained from \cite[Theorem~1]{chen2023ranking} by using our criterion \eqref{chenc} to verify the compatibility condition \cite[Condition~(8)]{chen2023ranking}. In the present setup, this criterion yields the sufficient condition $L_n\gtrsim(\log n)^3$. Alternatively, applying \cite[Theorem~3]{yang2024top} with unit likelihood weights yields the sufficient condition $L_n\gtrsim n^4(\log n)^2$. That theorem builds on \cite{chen2023ranking} and provides explicit conditions involving the spectral gap and maximum weighted degree. Both sufficient conditions are more stringent than the assumption $L_n=\omega((\log n)^2)$ imposed here.
\end{remark}

The proof is based on a direct computation using the asymptotic arguments developed in the preceding proofs and is given in \Cref{sec:k}. By comparing the exact asymptotics in \eqref{bipartite:expansion} and the upper bound \eqref{bipartite:bdd}, we see that the bound in \eqref{main:bdd} is tight up to logarithmic factors and that the global remainder term cannot in general be absorbed into the pair-specific effective-resistance term. Consequently, suitable balancing conditions may be necessary for asymptotic normality without debiasing. Further numerical illustration is provided in \Cref{sec:num}.

\section{Proofs of main results}\label{sec:proof}

\subsection{Proof of \Cref{thm:main}}\label{sec:main}

\subsubsection{Proof sketch}\label{sec:2}
Nearly all routes to uniform consistency rely on Taylor expansion at some stage. A key step in this process is to localize the MLE to a suitable neighborhood of the ground truth; see, e.g., \cite{chen2022partial}. In our analysis, we introduce a surrogate likelihood function that coincides with $l$ in a neighborhood of $\bm u^*$ while remaining globally strongly concave on the feasible region $\{\bm u\in\R^n: \bm 1^\top\bm u = 0\}$, which is similar to the primal-dual witness method in high-dimensional statistics \cite{wainwright2009sharp}. In particular, let $R = \kappa+1$ and define 
\begin{align*}
\widetilde{\psi}(x) = \begin{cases}
\psi(-R) + \psi'(-R)(x+R) + \frac{1}{2}\psi''(-R)(x+R)^2& x<-R\\
\psi(x) & |x|\leq R\\
\psi(R) + \psi'(R)(x-R) + \frac{1}{2}\psi''(R)(x-R)^2& x>R
\end{cases}.
\end{align*}
Replacing $\psi$ by $\widetilde{\psi}$ in \eqref{mle}, we define the surrogate log-likelihood $\widetilde{l}$ and the corresponding MLE by
\begin{align*}
\widetilde{l}(\bm u) = \sum_{1\leq i<j\leq n}\left[Y_{ij}\bm x_{ij}^\top\bm u-m_{ij}\widetilde{\psi}(\bm x_{ij}^\top\bm u)\right],\qquad \widetilde{\bm u} \coloneqq \argmax_{\bm u: \bm 1^\top\bm u= 0}\ \widetilde{l}(\bm u).
\end{align*}
The unique existence of $\widetilde{\bm u}$ follows from the construction since $\widetilde{l}$ is strongly concave. Moreover, the strong concavity of $\widetilde{l}$ ensures that its Hessian restricted to $\{\bm u\in\R^n: \bm 1^\top\bm u = 0\}$ is uniformly bounded away from singularity regardless of $\widetilde{\bm u}$. 

Our strategy is to first show that $\widetilde{\bm u}$ is uniformly consistent for $\bm u^*$, i.e., $\|\widetilde{\bm u}-\bm u^*\|_\infty=o_{\P}(1)$. Since the dynamic range of $\bm u^*$ is bounded by $\kappa$, with high probability, the dynamic range of $\widetilde{\bm u}$ is bounded by $\kappa + 1$ for all large $n$. Consequently, $\nabla l(\widetilde{\bm u})=\nabla\widetilde{l}(\widetilde{\bm u})=\bm 0$. By the uniqueness of the MLE, it follows that $\uu=\widetilde{\bm u}$, which completes the proof.

To establish the uniform consistency of the surrogate MLE, we apply a perturbation analysis to $\widetilde{l}$. A first idea is to perturb along the segment connecting $\bm u^*$ and $\uuu$, i.e., $\bm u_t = (1-t)\bm u^*+t\uuu$. By the fundamental theorem of calculus, 
\begin{align*}
- \nabla \widetilde{l} (\bm u^*) &=\nabla \lll (\uuu) - \nabla\lll (\bm u^*) = \left(\int_0^1\nabla^2\lll (\bm u_t) \d t\right)(\uuu-\bm u^*).
\end{align*}
Here the (integrated) negative Hessian can be interpreted as a weighted graph Laplacian with null space $\mathrm{span}\{\bm 1\}$ (see \Cref{sec:proof}). Since $\bm 1^\top (\uuu-\bm u^*) = 0$, 
\begin{align*}
\uuu-\bm u^* = \left(\int_0^1-\nabla^2\lll (\bm u_t) \d t\right)^\dagger\nabla \widetilde{l}(\bm u^*). 
\end{align*}
The integrated Hessian depends on the random estimator $\uuu$ and is therefore dependent on $\nabla\widetilde{l}(\bm u^*)$, making a direct concentration argument difficult. To separate the leading random fluctuation from this nonlinear dependence, we freeze the Hessian at $\bm u^*$ and define the linearized estimator $\vv$ by
\begin{align}
\vv-\bm u^*=\bigl(-\nabla^2\widetilde{l}(\bm u^*)\bigr)^\dagger\nabla\widetilde{l}(\bm u^*).
\label{vvd}
\end{align}
The vector $\vv$ can be viewed as the result of applying a single Newton--Raphson step for maximizing $\widetilde{l}$, initialized at $\bm u^*$. Similar linearizations appear as the first step of the preconditioned-gradient-descent analyses in \cite{chen2023ranking,yang2024top}, where the entire trajectory from $\bm u^*$ to the MLE (under the original log-likelihood) is controlled by induction. Closing the induction requires the nonlinear remainder to be absorbed by the proposed error bound, leading to additional self-consistency conditions such as \cite[Condition~(8)]{chen2023ranking}. 

Here we directly analyze the perturbation of $\widetilde{l}$ between $\vv$ and $\uuu$. By the definition of $\vv$, the constant and linear terms cancel when $\nabla\widetilde{l}(\vv)$ is expanded around $\bm u^*$, leaving a remainder that is quadratic in $\vv-\bm u^*$. Since $\vv-\bm u^*$ is small by concentration, this second-order improvement allows us to control $\uuu-\vv$ without tracking the full optimization trajectory. Meanwhile, the linearization fluctuation $\vv-\bm u^*$ can be controlled using the effective resistance of $\G^*_n$. 

\subsubsection{Full proof}
\paragraph{Notation and conditioning.}
Fix $\G=\G^*_n$ and recall the incidence matrix $\bm Q$ defined in \eqref{Q}. To see how the graph structure enters the analysis, we compute the gradient and Hessian of $\lll$: 
\begin{equation}
\begin{aligned}
\partial_k \lll(\bm u) &=
\sum_{j\neq k} (Y_{kj}-m_{kj}\widetilde{\psi}'(\bm x_{kj}^\top\bm u)), \qquad k\in [n], \\
\partial_{kk'} \lll(\bm u) & = \begin{cases}
m_{kk'}\widetilde{\psi}''(\bm x_{kk'}^\top\bm u)& k\neq k'\\
-\sum_{j\neq k}m_{kj}\widetilde{\psi}''(\bm x_{kj}^\top\bm u)& k = k'
\end{cases}.
\end{aligned}
\end{equation}
Since the dynamic range of $\bm u^*$ is bounded by $\kappa<R$, $\widetilde{\psi}'(\bm x_{kj}^\top\bm u^*)= \psi'(\bm x_{kj}^\top\bm u^*) = \sigma(\bm x_{kj}^\top\bm u^*)$. Define
\begin{align}
\bm\xi\coloneqq(\xi_{ij})_{1\leq i<j\leq n}\in\R^{n\choose 2}, \qquad\text{where $\xi_{ij}\coloneqq Y_{ij}-m_{ij}\sigma(\bm x_{ij}^\top\bm u^*)$}. \label{myxi}
\end{align}
To evaluate $\nabla\lll(\bm u)$ at $\bm u^*$, grouping the terms by edges following the orientation of $\bm Q$, 
\begin{align}
\nabla\lll(\bm u^*)=\sum_{1\leq i<j\leq n}\xi_{ij}\bm x_{ij}=\bm Q\bm\xi.
\label{mc}
\end{align}
Since the components of $\bm\xi$ are independent and centered, $\bm\xi$ can be interpreted as random noise on the edges, while $\nabla\lll(\bm u^*)=\bm Q\bm\xi$ is its divergence at the vertices.

The negative Hessian $-\nabla^2\lll (\bm u)$ is a weighted graph Laplacian with weights $\M_n(\bm u) = \{m_{ij}\widetilde{\psi}''(\bm x_{ij}^\top\bm u)\}_{1\leq i<j\leq n}$:    
\begin{align}
-\nabla^2 \lll(\bm u) = \bm L_{\M_n(\bm u)} \coloneqq \bm Q\bm W(\bm u)\bm Q^\top, \qquad \bm W(\bm u)\coloneqq\text{diag}(\M_n(\bm u)). \label{mr}
\end{align}
When $\bm u = \bm u^*$, $\widetilde{\psi}''(\bm x_{ij}^\top\bm u) = \psi''(\bm x_{ij}^\top\bm u)=\sigma'(\bm x_{ij}^\top\bm u)$.
Thus, $\M_n^*=\M_n(\bm u^*)$, which coincides with the collection of edge weights defining $\G_n^*$.

For convenience, we use ordered pairs $e = (i, j)$, $1\leq i<j\leq n$ to index the rows and columns of $\bm W(\bm u)$. For any $\bm u\in\R^n$ and $e = (i, j)$, a direct computation shows that 
\begin{align}
\frac{m_e}{2(1+e^{\kappa+1})}\leq [\bm W(\bm u)]_{ee}\leq \frac{m_e}{4}. \label{w-comp}
\end{align}
By \Cref{lm:rp}, the spectral gap and effective resistance of networks with conductances $\M_n^*$ and $\M_n(\bm u)$ are equivalent up to constants depending only on $\kappa$. 

We now define a high-probability event conditioned on which the subsequent discussion is carried out. First note that $\lambda_2(\bm L_{\G^*_n})=\omega((\log n)^2)$ and Cheeger's inequality ($h_{\G^*_n}\geq\lambda_2(\bm L_{\G^*_n})/2$) \cite{mohar1989isoperimetric, mohar1993eigenvalues} imply that $h_{\G^*_n} = \omega((\log n)^2)$. Under this condition, the MLE $\uu$ uniquely exists and is finite with probability at least $1-n^{-3}$ for all large $n$. The proof is standard and follows from a concentration argument plus a union bound; see, e.g., \cite[Lemma~C.1]{han2023unified}. Moreover, the components of 
\begin{align}
\bm\eta\coloneqq [\bm W(\bm u^*)^\dagger]^{1/2}\bm\xi\label{eta}
\end{align}
indexed by compared edges are independent, centered, standardized binomial residuals whose sub-Gaussian norms are bounded by $C e^{\kappa/2}$ for an absolute constant $C>0$. By Hoeffding's inequality and a union bound, with probability at least $1-n^{-3}$, simultaneously for all $i, j\in[n]$,
\begin{equation}\label{noise}
\begin{aligned}
\left|\bm x^\top_{ij}\bm L_{\G_n^*}^\dagger\bm Q\bm W(\bm u^*)^{1/2}\bm\eta\right|
&\lesssim \left\|\bigl(\bm Q\bm W(\bm u^*)^{1/2}\bigr)^\top\bm L_{\G_n^*}^\dagger\bm x_{ij}\right\|_2\sqrt{e^{\kappa}\log n}\\
& = \sqrt{e^{\kappa}\rr_{\G_n^*}(i\leftrightarrow j)\log n}
\end{aligned}, 
\end{equation}
where $\rr_{\G_n^*}(i\leftrightarrow j)$ denotes the effective resistance between $i$ and $j$ in $\G_n^*$. 
By a union bound, both the MLE-existence event and \eqref{noise} hold with probability at least $1-n^{-2}$. In the following discussion, we condition on this event.

\paragraph{Analysis of $\vv-\bm u^*$.}

We first bound the pairwise difference of $\vv-\bm u^*$ using the effective resistance of $\G^*_n$. This step is similar to the analysis in \cite[Step 1 in Section 3.3]{chen2023ranking}. Recall the definition of $\vv$ in \eqref{vvd} and substituting $\nabla\lll (\bm u^*)$ and $\nabla^2\lll(\bm u^*)$ using \eqref{mc} and \eqref{mr}:  
\begin{align}
\vv-\bm u^* &= \bm L_{\G_n^*}^\dagger\bm Q\bm\xi= \bm L_{\G_n^*}^\dagger(\bm Q\bm W(\bm u^*)^{1/2})\bm\eta.\label{linear-error} 
\end{align}
For any $i, j\in [n]$, multiplying both sides on the left by $\bm x^\top_{ij}$ and applying \eqref{noise}, 
\begin{equation}\label{bdd1}
\begin{aligned}
|\widetilde{v}_i-u^*_i - (\widetilde{v}_j-u^*_j)| &= \left|\bm x_{ij}^\top\bm L_{\G_n^*}^\dagger(\bm Q\bm W(\bm u^*)^{1/2})\bm\eta\right|\lesssim \sqrt{e^{\kappa}\rr_{\G_n^*}(i\leftrightarrow j)\log n}.
\end{aligned}
\end{equation}
By $\lambda_2(\bm L_{\G^*_n})=\omega((\log n)^{2})$ and \eqref{transfer}, $\rr_{\G^*_n}(i\leftrightarrow j)\leq 2/\lambda_2(\bm L_{\G^*_n})=o((\log n)^{-2})$, so the upper bound in \eqref{bdd1} is $o(1)$. Hence, for all large $n$, the dynamic range of $\vv-\bm u^*$ is bounded by $1/2$, i.e., the dynamic range of $\bm v_t =  (1-t)\bm u^*+t\vv$ is strictly bounded by $R$. 

\paragraph{Analysis of $\uuu-\vv$.}

The analysis of $\uuu-\vv$ is where our approach departs from existing work. We first represent $\nabla \lll(\vv)$ by applying the fundamental theorem of calculus to $\lll$ along $\bm v_t =  (1-t)\bm u^*+t\vv$: 
\begin{align*}
\nabla \lll(\vv) - \nabla \lll(\bm u^*) = \left(\int_0^1\nabla^2\lll(\bm v_t)\d t\right)(\vv - \bm u^*).
\end{align*}
Combining this with \eqref{vvd} and the matrix representation of the integrand in \eqref{mr}, we have 
\begin{equation}
\begin{aligned}\label{myvv}
\nabla \lll(\vv) = \left(\int_0^1\nabla^2\lll(\bm v_t)-\nabla^2\lll(\bm u^*)\d t\right)(\vv - \bm u^*)= \bm Q\bm\Lambda\bm Q^\top (\vv - \bm u^*),
\end{aligned}
\end{equation}
where 
\begin{align*}
\bm\Lambda\coloneqq \int_0^1\bm W(\bm u^*)-\bm W(\bm v_t) \d t. 
\end{align*}

The diagonal entries of $\bm\Lambda$ concern the perturbation of the Hessian of $\lll$ along the path $\bm v_t$. For all large $n$, the dynamic range of $\bm v_t$ remains bounded by $R$. As a result, along the perturbation path $\bm v_t$, $\widetilde{\psi}=\psi$ is smooth ($\widetilde{\psi}$ is not globally $C^3$ due to truncation). For each $e=(i, j)$ with $i<j$ and $|z-\bm x_e^\top\bm u^*|\leq 1/2$,  
\begin{align*}
|\psi'''(z)|\leq \psi''(z)\leq 2\psi''(\bm x_e^\top\bm u^*). 
\end{align*}
By the mean-value theorem, 
\begin{align*}
|[\bm W(\bm u^*)]_{ee} - [\bm W(\bm v_t)]_{ee}| &= m_e|\widetilde{\psi}''(\bm x_e^\top\bm u^*)-\widetilde{\psi}''(\bm x_e^\top\bm v_t)|\\
& = m_e|\psi''(\bm x_e^\top\bm u^*)-\psi''(\bm x_e^\top\bm v_t)|\\
&\leq 2m_e\psi''(\bm x_e^\top\bm u^*)\cdot |\widetilde{v}_i-u^*_i - (\widetilde{v}_j-u^*_j)| t.
\end{align*}
Integrating both sides from $0$ to $1$,  
\begin{equation}\label{diagp}
\begin{aligned}
|[\bm\Lambda]_{ee}|\leq m_e\psi''(\bm x_e^\top\bm u^*)\cdot |\widetilde{v}_i-u^*_i - (\widetilde{v}_j-u^*_j)|\stackrel{\eqref{bdd1}}{\lesssim}  [\bm W(\bm u^*)]_{ee}\sqrt{e^{\kappa}\rr_{\G_n^*}(i\leftrightarrow j)\log n}. 
\end{aligned}
\end{equation}

To connect $\vv$ to $\uuu$, we consider the perturbation of $\lll$ along $\bm w_t = (1-t)\vv + t\uuu$ and apply the fundamental theorem of calculus. Since $\uuu$ is the maximizer of $\lll$, by the first-order optimality condition, $\nabla\lll(\uuu) = \bm 0$. Combining this with \eqref{mr}, 
\begin{equation}\label{214}
\begin{aligned}
\nabla \lll (\vv) = \nabla \lll (\vv) - \nabla \lll (\uuu) &= \left(\int_0^1-\nabla^2\lll(\bm w_t) \d t\right)(\uuu - \vv)\\
& = \bm Q\left(\int_0^1\bm W(\bm w_t) \d t\right)\bm Q^\top (\uuu - \vv). 
\end{aligned} 
\end{equation}
Thus, the integrated negative Hessian in \eqref{214} is a (random) weighted graph Laplacian. We denote the corresponding weighted graph by $\G_n'=([n], \mathcal C)$ and its edge weights by $\mathcal C=\{c_{ij}\}_{1\leq i<j\leq n}$, where $c_{ij}$ denotes the $(i, j)$th diagonal entry of $\int_0^1\bm W(\bm w_t)\d t$. The resulting graph Laplacian is denoted by $\bm L_{\G_n'}$. In particular, $\ker(\bm L_{\G_n'}) = \mathrm{span}\{\bm 1\}$. Since $\bm 1^\top(\uuu - \vv)=0$, substituting $\nabla \lll (\vv)$ in \eqref{214} using \eqref{myvv} and rearranging terms gives
\begin{equation}\label{sec}
\begin{aligned}
\uuu - \vv = \bm L_{\G_n'}^\dagger \nabla \lll (\vv)= \bm L_{\G_n'}^\dagger\bm Q\bm\Lambda\bm Q^\top (\vv - \bm u^*). 
\end{aligned}
\end{equation}
For any $i, j\in [n]$, multiplying on the left of both sides of \eqref{sec} by $\bm x^\top_{ij}$,  
\begin{equation}\label{pang1}
\begin{aligned}
&|(\widetilde{u}_i - \widetilde{v}_i)-(\widetilde{u}_j - \widetilde{v}_j)|\\
=&\ |\bm x_{ij}^\top\bm L_{\G_n'}^\dagger\bm Q\bm\Lambda\bm Q^\top (\vv - \bm u^*)|\\
=&\ \left|\sum_{1\leq k<\ell\leq n}(\bm x_{ij}^\top\bm L_{\G_n'}^\dagger\bm x_{k\ell})\cdot\bm\Lambda_{(k, \ell), (k, \ell)}\cdot \bm x_{k\ell}^\top(\vv - \bm u^*)\right|\\
\lesssim&\ \sum_{1\leq k<\ell\leq n}\left|\bm x_{k\ell}^\top\bm L_{\G_n'}^\dagger\bm x_{ij}\right|([\bm W(\bm u^*)]_{(k, \ell), (k, \ell)}e^{\kappa}\rr_{\G_n^*}(k\leftrightarrow\ell)\log n)\qquad\text{(by \eqref{bdd1} and \eqref{diagp})}\\
\lesssim&\ e^{\kappa}\max_{e}\frac{[\bm W(\bm u^*)]_{ee}}{c_e}\max_{e}\frac{\rr_{\G_n^*}(e)}{\rr_{\G_n'}(e)}\log n\underbrace{\sum_{1\leq k<\ell\leq n}\left|\bm x_{k\ell}^\top\bm L_{\G_n'}^\dagger\bm x_{ij}\right|(c_{k\ell}\rr_{\G'_n}(k\leftrightarrow\ell))}_{\eqqcolon Q_{ij}},
\end{aligned}
\end{equation}
where the last inequality follows from Lemma~\ref{lm:rp} applied to $\G_n^*$ and $\G_n'$. Note that $Q_{ij}$ involves a sum over all edges of $\G'_n$. A crucial step in our analysis is to derive an effective upper bound on $Q_{ij}$ in terms of global properties of $\G'_n$, as stated in the following lemma. 
\begin{lemma}\label{lm:tech}
For all $i, j\in [n]$, 
\begin{align}
Q_{ij}\leq\frac{12+8\log n}{\lambda_2(\bm L_{\G_n'})}.\label{key}
\end{align}
\end{lemma}
\Cref{lm:tech} is a pathwise statement; its proof is based on the electrical-network interpretation in \Cref{sec:en} and is deferred to \Cref{sec:chain}.  
To complete the proof, substituting \eqref{key} into \eqref{pang1} yields the following upper bound: 
\begin{align}
|(\widetilde{u}_i - \widetilde{v}_i)-(\widetilde{u}_j - \widetilde{v}_j)|\lesssim\ e^{\kappa}\max_{e}\frac{[\bm W(\bm u^*)]_{ee}}{c_e}\max_{e}\frac{\rr_{\G_n^*}(e)}{\rr_{\G_n'}(e)}\log n\cdot \frac{12+8\log n}{\lambda_2(\bm L_{\G_n'})} = o(1),\label{hjd}
\end{align}
where the last step follows from the assumption $\lambda_2(\bm L_{\G^*_n})=\omega((\log n)^2)$ and the fact that the conductance profiles of $\G'_n$ and $\G^*_n$ are equivalent up to a factor depending only on $\kappa$ (see, e.g., \eqref{w-comp}), with the convention $0/0=1$. Combining \eqref{bdd1} and \eqref{hjd} and applying the triangle inequality,  
\begin{align}
\max_{i, j}|(\widetilde{u}_i-\widetilde{u}_j)-(u_i^*-u_j^*)| = o(1). \label{temp}
\end{align}
Consequently, for all large $n$, the dynamic range of $\uuu$ is bounded by $R$. Since $\nabla\lll(\uuu) = \nabla l(\uuu) = \nabla l(\uu) = \bm 0$, by the uniqueness of the MLE, $\uu = \uuu$. This establishes the uniform consistency of the MLE. We now apply \eqref{pang1} in an additional bootstrapping argument to sharpen the dependence on $\kappa$ and eliminate the dependence on the random graph $\G_n'$.

To this end, note that \eqref{temp} and \eqref{bdd1} together imply $\sup_{0\leq t\leq 1}\|\bm w_t-\bm u^*\|_\infty = o(1)$. Consequently, for all large $n$, 
\begin{equation}\label{comp-2}
\begin{aligned}
\max_{e}\frac{[\bm W(\bm u^*)]_{ee}}{c_e} &= \max_{e}\frac{\widetilde{\psi}''(\bm x_e^\top\bm u^*)}{\int_0^1\widetilde{\psi}''(\bm x_e^\top\bm w_t) \d t}\\
&\leq \max_e\frac{\widetilde{\psi}''(\bm x_e^\top\bm u^*)}{\min_{|z-\bm x_e^\top\bm u^*|\leq 1/2}\widetilde{\psi}''(z)}\\
&= \max_e\frac{\psi''(\bm x_e^\top\bm u^*)}{\min_{|z-\bm x_e^\top\bm u^*|\leq 1/2}\psi''(z)}\\
&\leq \sup_{x}\frac{\sigma(x)(1-\sigma(x))}{\sigma(x+1/2)(1-\sigma(x+1/2))}\\
&\leq 2,  
\end{aligned}
\end{equation}
where the penultimate step uses that $\psi''(x) = \sigma(x)(1-\sigma(x))$ and it is symmetric and decreasing on $\R_+$. Similarly, 
\begin{align*}
\min_{e}\frac{[\bm W(\bm u^*)]_{ee}}{c_e}\geq\frac{1}{2}.  
\end{align*}
According to \Cref{lm:rp}, 
\begin{align}
\max\left\{\max_{e}\frac{[\bm W(\bm u^*)]_{ee}}{c_e}, \ \max_{e}\frac{\rr_{\G_n^*}(e)}{\rr_{\G_n'}(e)}, \ \frac{\lambda_2(\bm L_{\G_n^*})}{\lambda_2(\bm L_{\G_n'})}\right\}\leq 2. \label{cpd}
\end{align}
Substituting this into \eqref{hjd} yields the following simplification: 
\begin{align}
|(\widetilde{u}_i - \widetilde{v}_i)-(\widetilde{u}_j - \widetilde{v}_j)|\lesssim \frac{e^\kappa (\log n)^2}{\lambda_2(\bm L_{\G_n^*})}.\label{bdd2}
\end{align}
Combining \eqref{bdd1} and \eqref{bdd2} together yields
\begin{align}
|(\widehat{u}_i - u_i^*)-(\widehat{u}_j - u_j^*)|=|(\widetilde{u}_i - u_i^*)-(\widetilde{u}_j - u_j^*)|\lesssim \sqrt{e^{\kappa}\rr_{\G^*_n}(i\leftrightarrow j)\log n} + \frac{e^{\kappa} (\log n)^2}{\lambda_2(\bm L_{\G^*_n})}. 
\end{align}
This gives the desired estimate in \eqref{main:bdd}. 


\begin{remark}\label{rem:classical}
The global remainder in \eqref{main:bdd} admits an alternative bound in terms of the pair-specific effective resistance multiplied by $n$. By the maximum principle (see~\eqref{mp}) and \eqref{cpd}, 
\begin{align*}
\left|\bm x_{k\ell}^\top\bm L_{\G_n'}^\dagger\bm x_{ij}\right|\leq \rr_{\G_n'}(i\leftrightarrow j)\leq 2\rr_{\G^*_n}(i\leftrightarrow j). 
\end{align*}
This combined with the first inequality in \eqref{diagp} allows us to bound the second equality in \eqref{pang1} as follows: 
\begin{align*}
|(\widetilde{u}_i - \widetilde{v}_i)-(\widetilde{u}_j - \widetilde{v}_j)|&\lesssim \rr_{\G^*_n}(i\leftrightarrow j)\sum_{1\leq k<\ell\leq n}[\bm W(\bm u^*)]_{(k,\ell), (k, \ell)}|\bm x_{k\ell}^\top(\vv - \bm u^*)|^2\\
&= \rr_{\G^*_n}(i\leftrightarrow j)\cdot (\vv - \bm u^*)^\top\bm L_{\G_n^*}(\vv - \bm u^*)\\
& =  \rr_{\G^*_n}(i\leftrightarrow j)\cdot \bm\eta^\top\bm P\bm\eta\qquad\text{(by \eqref{linear-error})},
\end{align*}
where $\bm\eta$ is defined in \eqref{eta} and $\bm P\coloneqq\bm W(\bm u^*)^{1/2}\bm Q^\top\bm L_{\G_n^*}^\dagger\bm Q\bm W(\bm u^*)^{1/2}$. For compared edges, the entries of $\bm\eta$ are independent, centered, have unit variance, and have sub-Gaussian norms bounded by $Ce^{\kappa/2}$. Moreover, $\bm P$ is an orthogonal projection of rank
$n-1$, so $\E[\bm\eta^\top\bm P\bm\eta]=n-1$, $\|\bm P\|_2=1$, and $\|\bm P\|_F^2=n-1$. 
The Hanson--Wright inequality \cite{vershynin2018high} therefore gives, with probability at least $1-n^{-3}$,
\begin{align*}
\bm\eta^\top\bm P\bm\eta\leq n-1+Ce^\kappa\bigl(\sqrt{n\log n}+\log n\bigr)\lesssim e^\kappa n.
\end{align*}
We include this event in the initial high-probability event; the original $n^{-3}$ failure bounds ensure that the total
failure probability remains at most $n^{-2}$ for all large $n$. Consequently, 
\begin{align*}
|(\widetilde{u}_i - \widetilde{v}_i)-(\widetilde{u}_j - \widetilde{v}_j)|\lesssim e^\kappa n\rr_{\G^*_n}(i\leftrightarrow j). 
\end{align*}
\end{remark}


\subsection{Proof of \Cref{lm:tech}}\label{sec:chain}
To begin with, note that $c_{k\ell}\rr_{\G'_n}(k\leftrightarrow\ell)$ is the edge leverage score of $\G_n'$. Let $\bm v^{i\to j}\coloneqq\bm L_{\G_n'}^\dagger\bm x_{ij}$ denote the normalized voltage distribution given by injecting and extracting one unit current from $i$ and $j$ in $\G_n'$. For any $a\in\R$,  
\begin{align*}
Q_{ij} &= \sum_{1\leq k<\ell\leq n}\left|v^{i\to j}_k-v^{i\to j}_\ell\right|(c_{k\ell}\rr_{\G'_n}(k\leftrightarrow\ell))\\
&\leq \sum_{1\leq k<\ell\leq n}\left(\left|v^{i\to j}_k-a\right|+\left|v^{i\to j}_\ell-a\right|\right)(c_{k\ell}\rr_{\G'_n}(k\leftrightarrow\ell))\\
& = \sum_{k=1}^n \left|v^{i\to j}_k-a\right|\cdot\left(\sum_{\ell\neq k}c_{k\ell}\rr_{\G'_n}(k\leftrightarrow\ell)\right)\\
& = \sum_{k=1}^n \left|v^{i\to j}_k-a\right|\ell_{\G'_n}(k)\qquad\text{(by \eqref{lev})}\\
& = 2\left(1-\frac{1}{n}\right)\sum_{k=1}^n \left|v^{i\to j}_k-a\right| - \sum_{k=1}^n \left|v^{i\to j}_k-a\right|\bm e_k^\top\bm L_{\G_n'}\diag\mathrm{vec}(\bm L_{\G_n'}^{\dagger})\qquad\text{(by \eqref{lev1})}\\
& = 2\left(1-\frac{1}{n}\right)\sum_{k=1}^n \left|v^{i\to j}_k-a\right| - |\bm v^{i\to j}-a\bm 1|^\top\bm L_{\G_n'}\diag\mathrm{vec}(\bm L_{\G_n'}^{\dagger})\\
& = 2\left(1-\frac{1}{n}\right)\sum_{k=1}^n \left|v^{i\to j}_k-a\right| - \diag\mathrm{vec}(\bm L_{\G_n'}^{\dagger})^\top(\bm L_{\G_n'}|\bm v^{i\to j}-a\bm 1|),
\end{align*}
where $|\bm v^{i\to j}-a\bm 1|$ is the vector obtained by taking absolute value of $\bm v^{i\to j}-a\bm 1$ componentwise. 
Note that without the (componentwise) absolute value, $\bm L_{\G_n'}(\bm v^{i\to j}-a\bm 1) = \bm L_{\G_n'}\bm v^{i\to j} = \bm x_{ij}$, which has all components zero except at $i$ and $j$. Using this observation, we can bound $\bm L_{\G_n'}|\bm v^{i\to j}-a\bm 1|$ for the $k$th component as 
\begin{equation}\label{f1}
\begin{aligned}
\bm e_k^\top\bm L_{\G_n'}|\bm v^{i\to j}-a\bm 1| &= \sum_{\ell\neq k}c_{k\ell}(|v^{i\to j}_k - a| - |v^{i\to j}_\ell - a|)\\
&\leq \sum_{\ell\neq k}c_{k\ell}\sgn(v^{i\to j}_k - a)[(v^{i\to j}_k - a) - (v^{i\to j}_\ell - a)]\\
&= \sgn(v^{i\to j}_k - a)\bm e_k^\top\bm L_{\G_n'}(\bm v^{i\to j}-a\bm 1)\\
&=\begin{cases}
0 & k\notin\{i, j\}\\
\sgn(v^{i\to j}_i - a)\leq 1& k=i\\
-\sgn(v^{i\to j}_j - a)\leq 1& k = j
\end{cases}  
\end{aligned}
\end{equation}
Meanwhile, 
\begin{equation}\label{f2}
\begin{aligned}
0&=(\bm 1^\top\bm L_{\G_n'})|\bm v^{i\to j}-a\bm 1| \\
&= \bm 1^\top(\bm L_{\G_n'}|\bm v^{i\to j}-a\bm 1|)\\
& = \bm e_i^\top(\bm L_{\G_n'}|\bm v^{i\to j}-a\bm 1|) + \bm e_j^\top(\bm L_{\G_n'}|\bm v^{i\to j}-a\bm 1|) - \sum_{\ell\notin\{i, j\}}\left|\bm e_\ell^\top(\bm L_{\G_n'}|\bm v^{i\to j}-a\bm 1|)\right|.  
\end{aligned}
\end{equation}
Combining \eqref{f1} and \eqref{f2} gives 
\begin{align}
\left\|\bm L_{\G_n'}|\bm v^{i\to j}-a\bm 1|\right\|_1 \leq 4. \label{e1}
\end{align}
Meanwhile, 
\begin{align}
\|\diag\mathrm{vec}(\bm L_{\G_n'}^{\dagger})^\top\|_\infty\leq \|\bm L_{\G_n'}^{\dagger}\|_2 = \frac{1}{\lambda_2(\bm L_{\G_n'})}. \label{e2}
\end{align}
Substituting \eqref{e1} and \eqref{e2} into the estimates on $Q_{ij}$ and using H\"older's inequality, we have 
\begin{align}
Q_{ij}\leq 2\left(1-\frac{1}{n}\right)\sum_{k=1}^n \left|v^{i\to j}_k-a\right| + \frac{4}{\lambda_2(\bm L_{\G_n'})}. 
\end{align}
Choosing $a$ as any median of $\bm v^{i\to j}$ and dropping the $(1-1/n)$ factor gives 
\begin{align}
Q_{ij}\leq 2\sp(\bm v^{i\to j}) + \frac{4}{\lambda_2(\bm L_{\G_n'})},
\end{align}
where $\sp$ is the $\ell_1$-total deviation defined in \eqref{mysv}. Applying \Cref{lem:voltage-spread} to further bound $2\sp(\bm v^{i\to j})$ gives the desired bound.

\subsection{Proof of \Cref{prop:er}}\label{sec:prop}

We first explain the intuition behind the bound assuming $\lambda_2(\bm L_{\G_n})=\omega(1)$ (which is not needed later). Recall that $\bm v^{i\to j}=\bm L_{\G_n}^{\dagger}\bm x_{ij}$ is the voltage profile associated with injecting one unit of current at $i$ and extracting one unit at $j$, and $\rr_{\G_n}(i\leftrightarrow j)=\bm x_{ij}^\top\bm L_{\G_n}^{\dagger}\bm x_{ij}=(\bm v^{i\to j})^\top\bm L_{\G_n}\bm v^{i\to j}$. Since $m_{ij}\in\{0,1\}$ under the assumption $m_{ij}\leq 1$, $\G_n$ is an unweighted simple graph. When $\lambda_2(\bm L_{\G_n})=\omega(1)$, the graph is well connected, and one expects the unit current injected at $i$ to spread approximately evenly among the $d_i$ incident edges, so that each carries roughly $1/d_i$ units of current. Similarly, the current entering $j$ is approximately evenly distributed among its $d_j$ incident edges. If the voltage profile on the intermediate vertices is nearly flat, this suggests $v_i^{i\to j}\approx 1/d_i$ and $v_j^{i\to j}\approx-1/d_j$ and $v_k^{i\to j}\approx 0$. This leads us to consider the approximate (uncentered) voltage profile for $\bm v^{i\to j}$: 
\begin{align}
\bm v = \frac{1}{d_i}\bm e_i - \frac{1}{d_j}\bm e_j. 
\end{align}
In particular, its corresponding energy is 
\begin{align}
\bm v^\top\bm L_{\G_n}\bm v = \frac{1}{d_i}+\frac{1}{d_j} + \frac{2m_{ij}}{d_id_j},
\end{align}
which matches the leading-order term in both bounds in \eqref{locr} up to a correction.  

To establish the general result in rigor, we consider a perturbation between $\bm v^{i\to j}$ and $\bm v$. Denote the adjacency and diagonal degree matrices of $\G_n$ by $\bm A_{\G_n}$ and $\bm D_{\G_n}$, respectively. In particular, $\bm L_{\G_n} = \bm D_{\G_n} - \bm A_{\G_n}$. 
By definition, 
\begin{align*}
\bm v^{i\to j} = \bm L_{\G_n}^\dagger \bm x_{ij} = \bm L_{\G_n}^\dagger \bm D_{\G_n}\bm v = \bm y + \bm L_{\G_n}^\dagger\bm A_{\G_n}\bm v, 
\end{align*}
where $\bm y\coloneqq\bm v-\frac{1}{n}\bm 1\bm 1^\top\bm v$ is the centered $\bm v$. 
It follows from the direct computation that
\begin{equation}\label{hrj}
\begin{aligned}
\rr_{\G_n}(i\leftrightarrow j) &= (\bm v^{i\to j})^\top\bm L_{\G_n}\bm v^{i\to j}\\
&=(\bm y + \bm L_{\G_n}^\dagger\bm A_{\G_n}\bm v)^\top\bm L_{\G_n}(\bm y + \bm L_{\G_n}^\dagger\bm A_{\G_n}\bm v)\\
&=(\bm v + \bm L_{\G_n}^\dagger\bm A_{\G_n}\bm v)^\top\bm L_{\G_n}(\bm v + \bm L_{\G_n}^\dagger\bm A_{\G_n}\bm v)\qquad\text{(since $\ker(\bm L_{\G_n}) = \text{span}\{\bm 1\}$)}\\
&=\bm v^\top\bm L_{\G_n}\bm v + 2\bm v^\top\bm A_{\G_n}\bm v + (\bm A_{\G_n}\bm v)^\top\bm L_{\G_n}^\dagger(\bm A_{\G_n}\bm v)\\
&= \frac{1}{d_i}+\frac{1}{d_j} - \frac{2m_{ij}}{d_id_j} + (\bm A_{\G_n}\bm v)^\top\bm L_{\G_n}^\dagger(\bm A_{\G_n}\bm v)\\
&\leq \frac{1}{d_i}+\frac{1}{d_j} - \frac{2m_{ij}}{d_id_j}+ \frac{\|\bm A_{\G_n}\bm v\|_2^2}{\lambda_2(\bm L_{\G_n})}. 
\end{aligned}
\end{equation}
To proceed further, note that 
\begin{align*}
\|\bm A_{\G_n}\bm v\|_2^2 = \sum_{k=1}^n\left(\frac{m_{ki}}{d_i}-\frac{m_{kj}}{d_j}\right)^2\leq \sum_{k=1}^n\left[\left(\frac{m_{ki}}{d_i}\right)^2+\left(\frac{m_{kj}}{d_j}\right)^2\right] = \frac{1}{d_i}+\frac{1}{d_j}. 
\end{align*}
Substituting this into \eqref{hrj} yields the desired upper bound. The lower bound follows by dropping the nonnegative quadratic term in the last equality in \eqref{hrj}.

\subsection{Proof of \Cref{thm:norm}}\label{sec:an}

The proof of asymptotic normality builds on the proof in \Cref{sec:main}. For convenience, we adopt the same notation used there. Recall that $\uu$ and $\uuu$ are the MLE and surrogate MLE, respectively. Since condition \eqref{eq:an-pair-condition} implies \eqref{unib} in \Cref{thm:main} (see Remark~\ref{rem:1}), the proof in \Cref{sec:main} shows that $\uu=\uuu$ with probability tending to one. Therefore, it suffices to establish the asymptotic normality result for $\uuu$. From now on, we fix the pair $(i,j)=(i_n,j_n)$ and drop the subscript $n$ to reduce notational clutter.

Consider the $\sqrt{\rr_{\G_n^*}(i\leftrightarrow j)}$-normalized decomposition of the error of the surrogate MLE for the pair $(i,j)$ via the linearization $\vv$:
\begin{align*}
\frac{\widetilde{u}_i-u_i^*-(\widetilde{u}_j-u_j^*)}
{\sqrt{\rr_{\G_n^*}(i\leftrightarrow j)}}
&=
\frac{\widetilde{v}_i-u_i^*-(\widetilde{v}_j-u_j^*)}
{\sqrt{\rr_{\G_n^*}(i\leftrightarrow j)}}+
\frac{(\widetilde{u}_i-\widetilde{v}_i)
-(\widetilde{u}_j-\widetilde{v}_j)}
{\sqrt{\rr_{\G_n^*}(i\leftrightarrow j)}}.
\end{align*}
The two terms on the right-hand side correspond to the linearized error and the nonlinear remainder, whose magnitudes are bounded in \eqref{bdd1} and \eqref{bdd2}, respectively, with probability tending to one. In fact, combining \eqref{bdd2} and \eqref{eq:an-pair-condition} yields
\begin{align*}
\frac{
\left|
(\widetilde{u}_i-\widetilde{v}_i)
-(\widetilde{u}_j-\widetilde{v}_j)
\right|
}
{\sqrt{\rr_{\G_n^*}(i\leftrightarrow j)}}
&\lesssim
\frac{e^\kappa(\log n)^2}
{\lambda_2(\bm L_{\G_n^*})
\sqrt{\rr_{\G_n^*}(i\leftrightarrow j)}}=o(1).
\end{align*}
This shows that the second term is $o_{\P}(1)$.

As a result, it suffices to show that the first term converges in distribution to a standard normal random variable. To this end, recall from \eqref{linear-error} that
\begin{align}
\widetilde{v}_i-u_i^*-(\widetilde{v}_j-u_j^*)=\bm x_{ij}^\top\bm L_{\G_n^*}^\dagger\bm Q\bm\xi,
\end{align}
which is a sum of independent mean-zero random variables according to the definition of $\bm\xi$. Moreover, its variance is given by
\begin{align*}
&\bm x_{ij}^\top\bm L_{\G_n^*}^\dagger\bm Q\E[\bm\xi\bm\xi^\top]\bm Q^\top\bm L_{\G_n^*}^\dagger\bm x_{ij}=\bm x_{ij}^\top\bm L_{\G_n^*}^\dagger\bm Q\bm W(\bm u^*)\bm Q^\top\bm L_{\G_n^*}^\dagger\bm x_{ij}=\bm x_{ij}^\top\bm L_{\G_n^*}^\dagger\bm x_{ij}=\rr_{\G_n^*}(i\leftrightarrow j).
\end{align*}

To establish asymptotic normality, we apply the Lindeberg central limit theorem. For the $e=(i',j')$ component of $\bm\xi$, where $i'<j'$ and $m_e>0$, write $\xi_e=\sum_{\ell=1}^{m_e}\xi_{e,\ell}$, where the $\xi_{e,\ell}$ are independent centered Bernoulli random variables corresponding to the $\ell$th comparison on $e$. Thus,
\begin{align*}
\bm x_{ij}^\top\bm L_{\G_n^*}^\dagger\bm Q\bm\xi=\sum_{e:m_e>0}\sum_{\ell=1}^{m_e}\left(\bm x_{ij}^\top\bm L_{\G_n^*}^\dagger\bm x_e\right)\xi_{e,\ell}.
\end{align*}
Since $|\xi_{e,\ell}|\leq 1$, the maximum absolute magnitude of the normalized summands is bounded by
\begin{align*}
\max_{e:m_e>0}\max_{\ell\leq m_e}\frac{\left|\bm x_{ij}^\top\bm L_{\G_n^*}^\dagger\bm x_e\right||\xi_{e,\ell}|}
{\sqrt{\rr_{\G_n^*}(i\leftrightarrow j)}}\leq\max_{e:m_e>0}\frac{\left|\bm x_{ij}^\top\bm L_{\G_n^*}^\dagger\bm x_e\right|}
{\sqrt{\rr_{\G_n^*}(i\leftrightarrow j)}}\leq\max_{e:m_e>0}\sqrt{\rr_{\G_n^*}(e)}\leq\sqrt{\frac{2}{\lambda_2(\bm L_{\G_n^*})}}=o(1),
\end{align*}
where the second inequality follows from the Cauchy--Schwarz inequality and the third uses \eqref{transfer}. Therefore, for every fixed $\varepsilon>0$, the absolute magnitude of every normalized summand is eventually smaller than $\varepsilon$. This verifies the Lindeberg condition. Consequently,
\begin{align*}
\frac{\widetilde{v}_i-u_i^*-(\widetilde{v}_j-u_j^*)}{\sqrt{\rr_{\G_n^*}(i\leftrightarrow j)}}\Rightarrow \mathcal N(0,1).
\end{align*}
Combining this with the $o_{\P}(1)$ bound for the nonlinear remainder and recalling that $\uu=\uuu$ with probability tending to one proves \eqref{eq:an-oracle-clt}.

To further show that \eqref{eq:an-oracle-clt} holds for the plug-in estimator of $\rr_{\G_n^*}(i\leftrightarrow j)$, denote the plug-in analogue of $\G_n^*$ as $\widehat{\G}_n$, which has edge weights $\{m_{ij}\psi''(\bm x_{ij}^\top\uu)\}_{1\leq i<j\leq n}$.   
The uniform consistency of the MLE implies $\Delta_n = \max_{i, j}|(\widehat{u}_i-\widehat{u}_j)-(u_i^*-u_j^*)|=o_{\P}(1)$. It follows from a similar computation as the last two steps in \eqref{comp-2} that, for all large $n$, uniformly over all edges with $m_e>0$, 
\begin{align*}
1-2\Delta_n\leq\frac{m_e\psi''(\bm x_e^\top\uu)}{m_e\psi''(\bm x_e^\top\bm u^*)}\leq 1+2\Delta_n.
\end{align*}
Applying \Cref{lm:rp} to $\G^*_n$ and $\widehat{\G}_n$ gives
\begin{align}
\max_{1\leq i<j\leq n}\left|\frac{\rr_{\widehat{\G}_n}(i\leftrightarrow j)}{\rr_{\G_n^*}(i\leftrightarrow j)}-1\right|
\leq \frac{2\Delta_n}{1-2\Delta_n}=o_{\P}(1).
\end{align}
The proof is completed by invoking Slutsky's theorem.

\subsection{Proof of \Cref{thm:ger}}\label{sec:ger}
It suffices to estimate the effective resistances and the spectral gap of $\G_n$. We first obtain a lower bound on the latter using Rayleigh's monotonicity principle. Consider the natural coupling obtained by assigning to each ordered pair $(i, j)$ an independent random variable $U_{ij}^{(n)}\sim\operatorname{Unif}[0,1]$. Define
\begin{align*}
B_{ij}^{(n)}=\mathbb I\{U_{ij}^{(n)}\le p_{\min}^{(n)}\}\qquad \leq \qquad A_{ij}^{(n)}= \mathbb I\{U_{ij}^{(n)}\le p_{ij}^{(n)}\}.
\end{align*}
The graph $\HH_n$ with edges $B_{ij}^{(n)}$ is a standard Erd\H{o}s--R\'enyi model $G(n, p_{\min}^{(n)})$. Since $m_{ij}^{(n)}\geq B_{ij}^{(n)}$, \Cref{lm:rp} yields $\lambda_2(\bm L_{\G_n})\geq\lambda_2(\bm L_{\HH_n})$. Since $np_{\min}^{(n)}/\log n\to\infty$, standard spectral-gap estimates for Erd\H{o}s--R\'enyi random graphs give $\lambda_2(\bm L_{\HH_n})\geq np_{\min}^{(n)}/2$ for all large $n$ a.s.; see, e.g., \cite[Lemma 10]{MR3953449}. The assumption $np_{\min}^{(n)}=\omega((\log n)^2)$ thus implies $\lambda_2(\bm L_{\G_n})\geq np_{\min}^{(n)}/2=\omega((\log n)^2)$. 

We next estimate the effective resistances. Under the condition $np^{(n)}_{\min}/\log n\to\infty$, a standard concentration argument based on the Chernoff bound and a union bound shows that with probability tending to one, simultaneously for all $i\in [n]$, the degree of $i$ is close to its mean, i.e., 
\begin{align*}
\frac{1}{2}\bar{d}_i^{(n)}\leq d_i^{(n)}\leq 2\bar{d}_i^{(n)}. 
\end{align*}
This combined with the lower bound on $\lambda_2(\bm L_{\G_n})$ and \Cref{prop:er} shows that, for all large $n$, simultaneously for all $i\neq j$, 
\begin{align*}
\frac{1}{4\min\{\bar{d}_i^{(n)}, \bar{d}_j^{(n)}\}}\leq\rr_{\G_n}(i\leftrightarrow j)\leq \frac{4}{\min\{\bar{d}_i^{(n)}, \bar{d}_j^{(n)}\}}. 
\end{align*}
Further assuming $(np_{\min}^{(n)})^2/\bar{d}^{(n)}_{\max}=\omega((\log n)^4)$, 
\begin{align*}
\min_{1\leq i< j\leq n}\lambda_2^2(\bm L_{\G_n})\rr_{\G_n}(i\leftrightarrow j)\gtrsim \frac{(np_{\min}^{(n)})^2}{\bar{d}^{(n)}_{\max}}=\omega((\log n)^4).
\end{align*}
Under Assumption~\ref{ass:bdd}, the effective resistance and spectral gap estimates for $\G_n$ can be transferred to $\G_n^*$ using \eqref{yh} and \Cref{lm:rp}. The results in \eqref{random-bdd} and \eqref{random-an} thus follow by applying \Cref{thm:main} and \Cref{thm:norm}.

\subsection{Proof of \Cref{thm:bipartite}}\label{sec:k}

Let $\mathcal K_{2,n-2}$ denote the corresponding unweighted graph, whose spectral gap $\lambda_2(\bm L_{\mathcal K_{2,n-2}})=2$; see, e.g., \cite{brouwer2012spectra}. Therefore, $\lambda_2(\bm L_{\G_n})=L_n\lambda_2(\bm L_{\mathcal K_{2,n-2}})=2L_n=\omega\bigl((\log n)^2\bigr)$.
The degrees in $\mathcal K_{2,n-2}$ are $d_1=d_2=n-2$ and $d_i=2$ otherwise. Applying \Cref{prop:er} to $\mathcal K_{2,n-2}$ and using the resistance scaling $\rr_{\G_n}(i\leftrightarrow j)=\rr_{K_{2, n-2}}(i\leftrightarrow j)/L_n$ gives
\begin{align}\label{329}
\rr_{\G_n}(1\leftrightarrow2)
\lesssim \frac{1}{(n-2)L_n},
\qquad
\max_{\{i,j\}\neq\{1,2\}}
\rr_{\G_n}(i\leftrightarrow j)
\lesssim \frac{1}{L_n}.
\end{align}
Under Assumption~\ref{ass:bdd}, the effective resistance and spectral gap estimates for $\G_n$ can be transferred to $\G_n^*$ using \eqref{yh} and \Cref{lm:rp}. 
The upper bound in \eqref{bipartite:bdd} follows by invoking \eqref{main:bdd} in \Cref{thm:main}.

For the special parameter choice, write $N=n-2$ and $L=L_n$, and recall the surrogate MLE $\uuu$ and the linearization $\vv$ from \Cref{sec:main}. Here every edge of $\G_n^*$ has weight $2L/9$, so the symmetry gives
\[
\lambda_2(\bm L_{\G_n^*})=\frac{4L}{9},
\qquad
\rr_{\G_n^*}(1\leftrightarrow2)=\frac{9}{NL}.
\]
The Lindeberg argument in \Cref{sec:an} therefore gives
\[
Z_n\coloneqq\frac{\sqrt{NL}}{3}
\bigl[(\widetilde v_1-\widetilde v_2)-(u_1^*-u_2^*)\bigr]
\Rightarrow\mathcal N(0,1).
\]
Indeed, that argument for the linearized term requires only a diverging spectral gap; the balancing condition is used there only
to control the nonlinear remainder. Here we do not need the balancing condition and explicitly compute this remainder leveraging the special structure of the graph.

To analyze $\uuu-\vv$, start from the second equality in \eqref{pang1}. Write $d_{rk}=\widetilde v_r-\widetilde v_k-(u_r^*-u_k^*)$ and $\Lambda_{rk}=[\bm\Lambda]_{(r,k),(r,k)}$ for $r\in\{1,2\}$ and $3\leq k\leq n$. Then
\begin{align*}
(\widetilde u_1-\widetilde v_1)
-(\widetilde u_2-\widetilde v_2)
=\sum_{k=3}^n\sum_{r=1}^2
\bigl(\bm x_{12}^\top\bm L_{\G_n'}^\dagger\bm x_{rk}\bigr)
\Lambda_{rk}d_{rk}.
\end{align*}
The definition of $\bm\Lambda$ and Taylor expansion give
\[
\Lambda_{rk}d_{rk}
=\left(L\int_0^1\!\bigl[\sigma'(u_r^*)
-\sigma'(u_r^*+t d_{rk})\bigr]\d t\right)d_{rk}
=-\frac L2\sigma''(u_r^*)d_{rk}^2
+o_{\P}(L d_{rk}^2),
\]
uniformly over the edges, since $\max_{r,k}|d_{rk}|=o_{\P}(1)$
by \eqref{bdd1}.

By the uniform consistency of both $\uuu$ and $\vv$, the conductances of $\G_n'$ satisfy
$c_{rk}=(1+o_{\P}(1))\,2L/9$ uniformly. Under unit current from vertex $1$ to vertex $2$, the graph consists of $N$ parallel paths $1\to k\to2$. Each path has resistance
$c_{1k}^{-1}+c_{2k}^{-1}=(1+o_{\P}(1))\,9/L$, so each carries current $(1+o_{\P}(1))/N$, uniformly in $k$. Thus, writing $\bm v^{1\to 2}=\bm L_{\G_n'}^\dagger\bm x_{12}$, Ohm's law gives
\[
\bm x_{12}^\top\bm L_{\G_n'}^\dagger\bm x_{rk}
=v^{1\to 2}_r-v^{1\to 2}_k
=(-1)^{r-1}(1+o_{\P}(1))\frac9{2NL},
\qquad r=1,2,
\]
uniformly in $k$. Since $\sigma''(u_1^*)=-2/27$ and
$\sigma''(u_2^*)=2/27$, substitution into the correction identity yields
\[
(\widetilde u_1-\widetilde v_1)
-(\widetilde u_2-\widetilde v_2)
=\frac1{6N}\sum_{k=3}^n\sum_{r=1}^2 d_{rk}^2
+o_{\P}\!\left(\frac1N\sum_{k=3}^n\sum_{r=1}^2 d_{rk}^2\right).
\]

To evaluate the quadratic term, set $S_k=\xi_{1k}+\xi_{2k}$, with $\bm\xi$ as in \eqref{myxi}. The equations in \eqref{linear-error} at vertices $k\ge3$, together with the linearized fluctuation computed above, give
\[
d_{1k}+d_{2k}=\frac{9S_k}{2L},
\qquad
d_{1k}-d_{2k}=\frac{3Z_n}{\sqrt{NL}}.
\]
Using $a^2+b^2=[(a+b)^2+(a-b)^2]/2$, we obtain
\[
\frac1N\sum_{k=3}^n\sum_{r=1}^2 d_{rk}^2
=\frac{9Z_n^2}{2NL}
+\frac{81}{8NL^2}\sum_{k=3}^n S_k^2
=\frac9{2L}+O_{\P}\!\left(\frac1{L\sqrt N}\right).
\]
The last equality follows from $Z_n=O_{\P}(1)$ and the independence of the $S_k$, together with
$\E[S_k^2]=4L/9$ and $\E[S_k^4]=O(L^2)$. Consequently,
\[
(\widetilde u_1-\widetilde v_1)
-(\widetilde u_2-\widetilde v_2)
=\frac3{4L}+o_{\P}(L^{-1}).
\]
Combining this with the linearized fluctuation and using
$\P(\uu=\uuu)\to1$ proves \eqref{bipartite:expansion}.

\section{Numerical simulation}\label{sec:num}

In this section, we conduct numerical simulations to verify our theoretical findings. We consider two synthetic comparison designs, including the generalized Erd\H{o}s--R\'enyi graphs in \Cref{thm:ger} and the imbalanced complete bipartite graphs in \Cref{thm:bipartite}, as well as a real-world comparison network of ATP tennis players. We focus on the uniform consistency of the MLE, the tightness of the bound in \eqref{main:bdd}, and the asymptotic normality for a prescribed sequence of pairs. For each design below, we keep the comparison graph and the true utilities fixed and average over 1,000 independent trials simulated under the BT model.

\subsection{Generalized Erd\H{o}s--R\'enyi graphs}

We first generate Erd\H{o}s--R\'enyi graphs with $p_{ij}^{(n)}=0.75 n^{-1/2}$ and at most one comparison per edge. The utilities are drawn independently from $\operatorname{Unif}(-1.5,1.5)$ and then centered. For $n\in\{2^{11},\ldots,2^{14}\}$, the left panel of \Cref{fig:er-homogeneous} shows that the average of the maximum uniform pairwise error $\Delta_n$ steadily decreases from $4.08$ to $2.31$. This trend is consistent with the uniform consistency result in \Cref{thm:ger}. At each $n$, we sample 500 item pairs and plot their mean absolute errors against $\sqrt{\rr_{\G^*_n}(i\leftrightarrow j)}$. The increasing trend illustrates the linear dependence of the pairwise error on the square root of the pair-specific effective resistance, which is in agreement with the bound in \eqref{random-bdd} (as the linearization term dominates).

To assess the normal approximation, for the 500 selected pairs at each $n$, we construct the confidence intervals
\begin{equation}\label{confidence_interval}
\left[
(\widehat u_i-\widehat u_j)-1.96\sqrt{\rr_{\widehat{\G}_n}(i\leftrightarrow j)},
\,
(\widehat u_i-\widehat u_j)+1.96\sqrt{\rr_{\widehat{\G}_n}(i\leftrightarrow j)}
\right],
\end{equation}
where $\widehat{\G}_n$ denotes the plug-in estimate of $\G_n^*$ obtained using the MLE. The average coverage probability is close to the nominal level of $95\%$. This coincides with our theoretical prediction based on \Cref{thm:norm,thm:ger}.

\begin{figure}[hbt!]
\centering
\includegraphics[width=\textwidth]{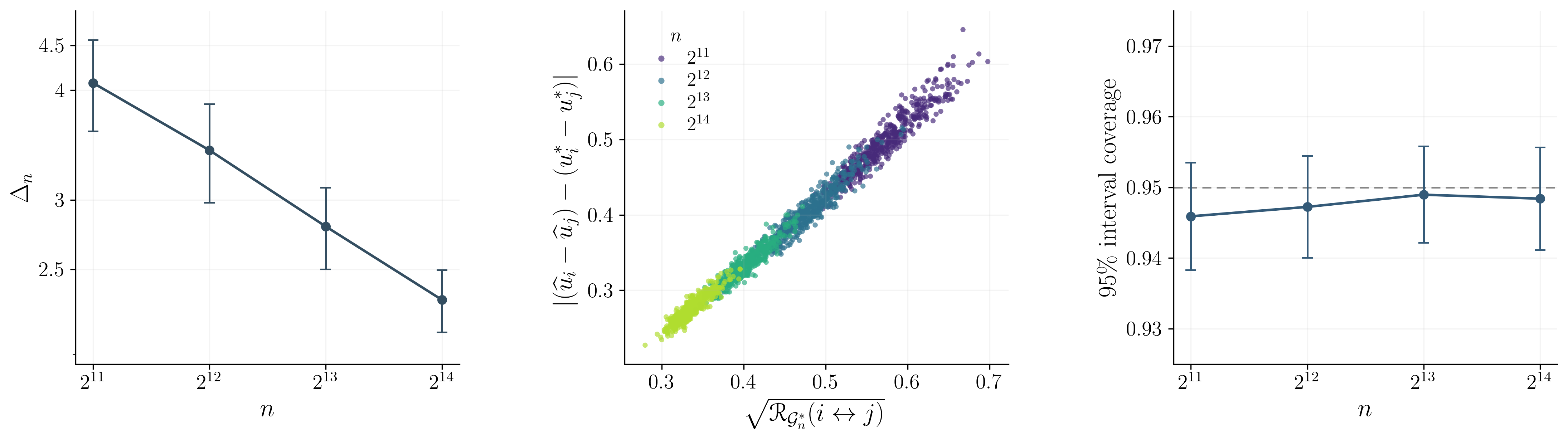}
\caption{Left: $\Delta_n$ under different choices of $n$ with one-standard-deviation whiskers. Middle: Pairwise error versus square-root effective resistance at different $n$. Right: Coverage probabilities for pairwise differences, with whiskers across pairs. The grey dashed line is the normal benchmark for nominal coverage.}\label{fig:er-homogeneous}
\end{figure}

To examine how heterogeneity in the comparison graph influences asymptotic normality, we modify the graph-generating process by creating a small set of locally dense vertices. Fix $n=2^{13}$ and let $\mathcal{V}_0$ be a fixed set of 20 vertices. We repeat the simulation over modified graphs with $p_{ij}^{(n)}=p_{\mathcal{V}_0}$ for all $i\in\mathcal{V}_0$ and $j\in[n]$, while leaving all other edge probabilities unchanged. The modified edge probability $p_{\mathcal{V}_0}$ takes values in $\{0.75 n^{-1/2}, 0.04, 0.08, 0.16, 0.32, 0.8\}$. The results are shown in Figure \ref{fig:er-heterogeneous}. As $p_{\mathcal{V}_0}$ increases from $0.75 n^{-1/2}$ to $0.8$, the mean absolute pairwise error within $\mathcal{V}_0$ decreases from $0.34$ to $0.039$, while the average coverage rate falls substantially. In contrast, coverage for pairs outside $\mathcal{V}_0$ remains close to $0.95$. This suggests that substantially increasing graph heterogeneity through higher local comparison rates does not impair consistency for vertices in $\mathcal{V}_0$, but can make the normal approximation less accurate. This observation also illustrates the role of the balancing condition in \Cref{thm:ger}~(ii) in the absence of a further debiasing procedure.

\begin{figure}[hbt!]
\centering
\includegraphics[width=\textwidth]{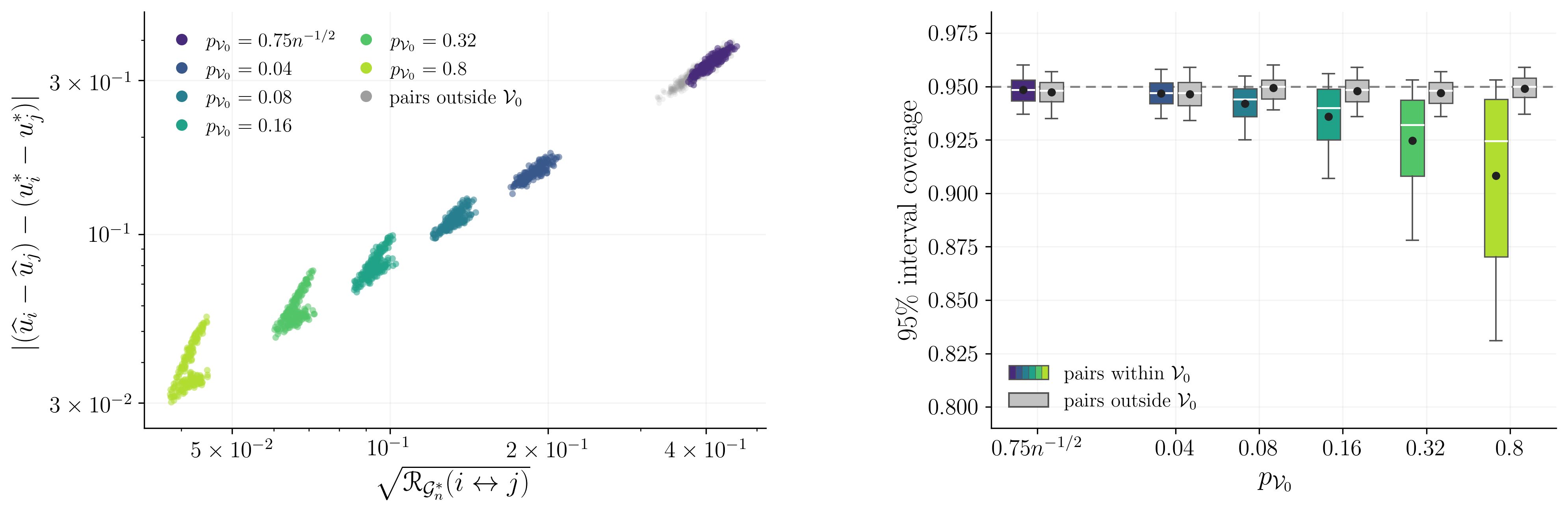}
\caption{Left: Mean absolute errors in estimated pairwise utility differences for pairs within $\mathcal{V}_0$ at $n=2^{13}$, plotted against square-root effective resistance. Right: Coverage probabilities for pairwise utility differences within and outside $\mathcal{V}_0$ under different edge-forming probabilities.}\label{fig:er-heterogeneous}
\end{figure}

\subsection{Imbalanced complete bipartite graphs}
\Cref{thm:bipartite} shows that the error bound in \eqref{main:bdd} is tight for certain graph sequences. In this experiment, we numerically verify this theoretical prediction. Let $\G_n=L\mathcal K_{2,n-2}$, with $u_1^*=-u_2^*=\log 2$ and $u_k^*=0$ for $k\geq 3$. For the pair $\{1,2\}$, we have $\rr_{\G_n}(1\leftrightarrow 2)=2/((n-2)L)$ and $\lambda_2(\bm L_{\G_n})=2L$. We first fix $n$ and vary $L$ to examine how the linearization error and the nonlinear remainder jointly determine the pairwise estimation error between the objects $1$ and $2$. We then further vary $n$ under prescribed choices of $L$ to test the predicted rates along different graph sequences.

Let $\Delta_{12}=|(\widehat u_1-\widehat u_2)-(u_1^*-u_2^*)|$. The expansion \eqref{bipartite:expansion} suggests the approximation of this error as
\begin{equation}\label{exact_error} m_{n,L}:=\mathbb E\left[\left|
\frac{3}{\sqrt{(n-2)L}}Z+\frac{3}{4L}
\right|\right],
\qquad Z\sim \mathcal N(0,1).
\end{equation}

We first fix $n=2^{11}$ and vary $L$. For the pair $\{1,2\}$, the expected absolute magnitudes of the linearization term and the nonlinear remainder term in \eqref{exact_error} are, respectively,
\begin{align*}
c_1\sqrt{\rr_{\G_n}(1\leftrightarrow 2)}
\qquad\text{and}\qquad
\frac{c_2}{\lambda_2(\bm L_{\G_n})},
\end{align*}
where $c_1=3/\sqrt{\pi}$ and $c_2=3/2$. In particular, their contributions decay at rates $(nL)^{-1/2}$ and $L^{-1}$, respectively. They are equal at the crossover point
\begin{align*}
L_c=\frac{\pi(n-2)}{32}\approx 201.
\end{align*}
The left panel of \Cref{fig:two-item-crossover} shows that the pairwise error $\Delta_{12}$ closely follows the predicted error scale $m_{n,L}$ across this crossover. When $L\ll L_c$, the nonlinear remainder dominates, and the error decays at the spectral-gap rate $L^{-1}$. When $L\gg L_c$, the linearization term dominates, and the error decays at the effective-resistance rate $(nL)^{-1/2}$. The right panel shows the same crossover for different values of $n$ after rescaling $L$ by $L_c$ and the error by $c_1\sqrt{\rr_{\G_n}(1\leftrightarrow 2)}$. The resulting agreement across the two regimes illustrates the sharpness of the error bound for this design and shows that both the effective-resistance and spectral-gap terms are needed to capture the estimation error.

\begin{figure}[hbt!]
\centering
\includegraphics[width=.98\textwidth]{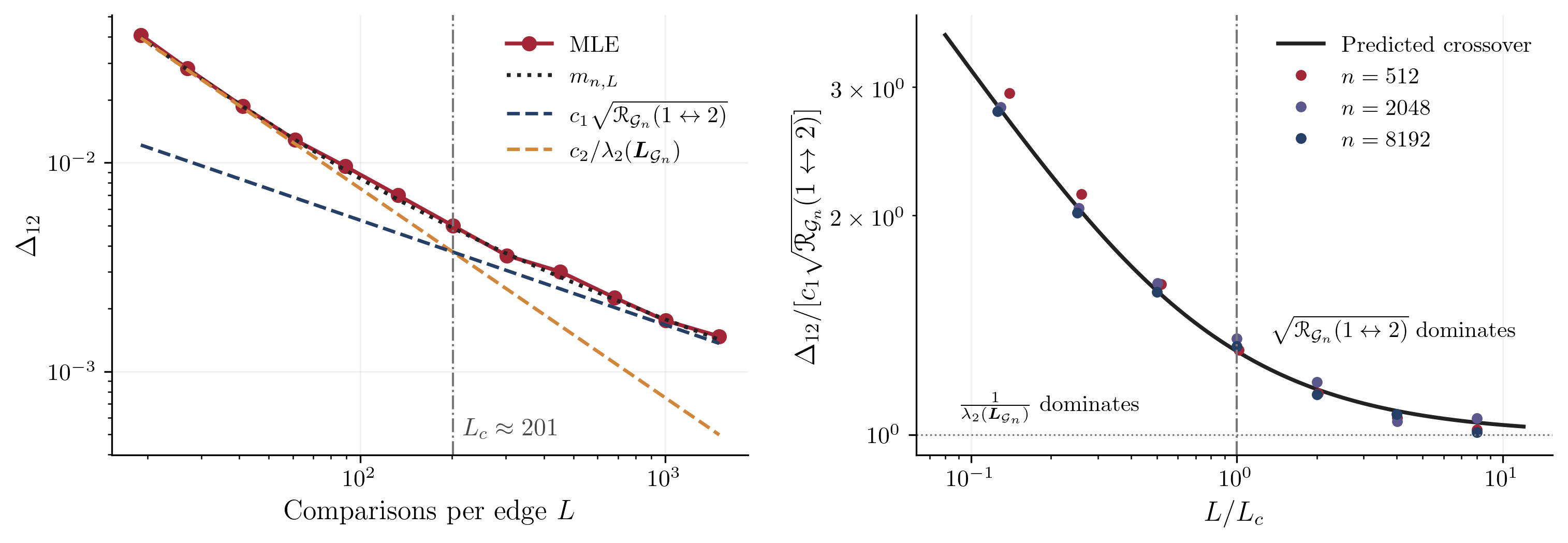}
\caption{Mean absolute error for the pair $\{1,2\}$ as the number of comparisons per edge varies.
Left: Averaged $\Delta_{12}$ versus $L$ at fixed $n=2^{11}$. The dashed curves show $c_1\sqrt{\rr_{\G_n}(1\leftrightarrow2)}$ and $c_2/\lambda_2(\bm L_{\G_n})$; the dotted curve shows $m_{n,L}$. Right: Normalized mean absolute error versus $L/L_c$ for
different values of $n$, with the common normalized prediction.} 
\label{fig:two-item-crossover}
\end{figure}

We next consider different scalings of $L$ as $n$ increases, and write $L_n$ to make this dependence explicit. For $L_n=\lceil n^{1/2}\rceil$, the nonlinear remainder is of order $n^{-1/2}$, while the linearization error is of the smaller order $n^{-3/4}$. For $L_n=\lceil n^{3/2}\rceil$, the linearization error is of order $n^{-5/4}$, while the nonlinear remainder is of the smaller order $n^{-3/2}$. When $L_n=\lceil 0.1n\rceil$, both terms are of order $n^{-1}$. The observed errors in \Cref{fig:two-item-sequences} follow these three predicted rates. Thus, each of the two graph-dependent terms in \eqref{main:bdd} can attain the leading order along an appropriate graph sequence. Taken together, these two experiments suggest that the estimation error bound in \eqref{main:bdd} is tight for general graphs up to logarithmic factors.

\begin{figure}[hbt!]
\centering
\includegraphics[width=.98\textwidth]{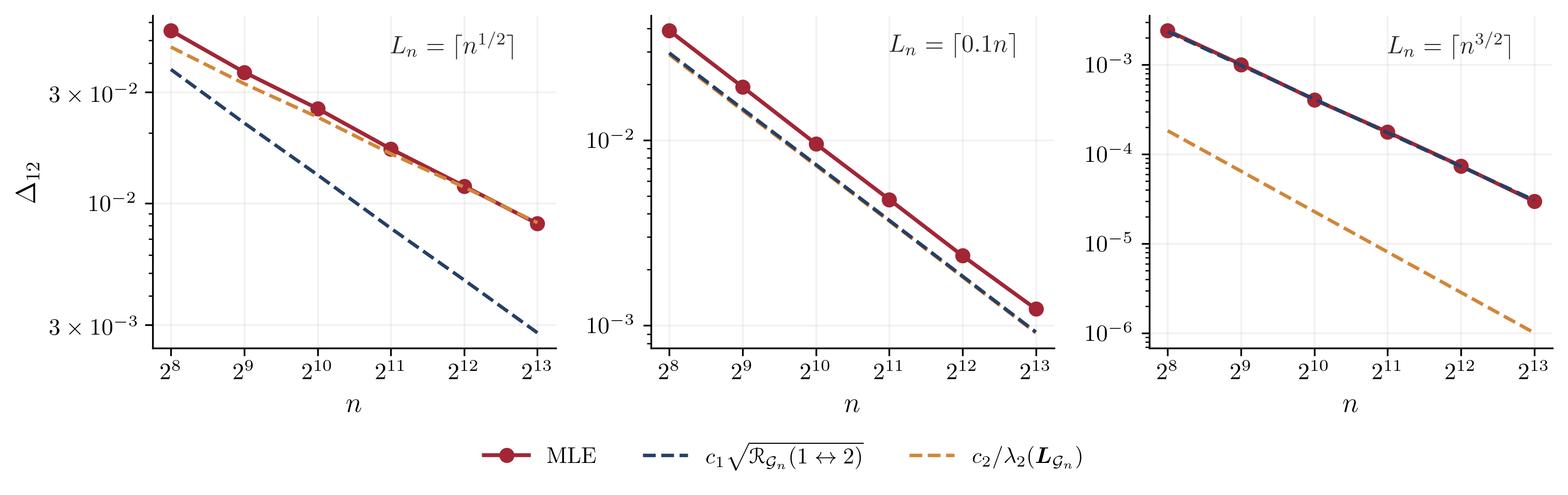}
\caption{Averaged absolute error for the pair $\{1,2\}$ as $n$ varies under three prescribed choices of $L_n$.
The dashed curves show $c_1\sqrt{\rr_{\G_n}(1\leftrightarrow2)}$ and $c_2/\lambda_2(\bm L_{\G_n})$, corresponding to
the linearized error and global remainder, respectively.}
\label{fig:two-item-sequences}
\end{figure}

\subsection{A tennis comparison network}

We consider a pairwise comparison network constructed from the publicly available ATP match data compiled and maintained by Jeff Sackmann through Tennis Abstract \citep{sackmann_tennis_atp}. We iteratively prune the graph to retain all players with at least 20 matches and keep the comparison outcomes strongly connected. After preprocessing the records from 1980--2024, we obtain a fixed comparison graph $\G_n$ with 1,076 players, 63,314 compared pairs, and 124,005 matches. We fit the BT model to the observed match outcomes and take the resulting MLE as the ground-truth utility vector $\bm u^*$, which has dynamic range $\kappa=4.96$. Keeping the observed comparison counts $m_{ij}$ fixed, we generate synthetic outcomes according to
\begin{align*}
Y_{ij}\sim\operatorname{Binomial}\!\left(m_{ij},\sigma(u_i^*-u_j^*)\right).
\end{align*}
This construction preserves the comparison structure and sampling heterogeneity of the ATP network while generating outcomes from a correctly specified BT model.

The left panel of \Cref{fig:tennis-era-structure} visualizes $\G_n$, with players colored according to their utilities. The comparison graph is highly unbalanced in degree, with degrees ranging from 21 to 1,503 and a median of 146. In contrast, the local geometry, as measured by the square root pairwise effective resistance, is more balanced, ranging from $0.095$ to $0.638$, with the central $90\%$ lying between $0.150$ and $0.448$, as shown in the right panel. The reciprocal spectral gap is $1/\lambda_2(\bm L_{\G_n^*})=0.494$. Thus, ignoring logarithmic factors, the linearization error and the global bound on the nonlinear remainder in \eqref{main:bdd} are of comparable order.

\begin{figure}[hbt!]
\centering
\includegraphics[width=.98\textwidth]{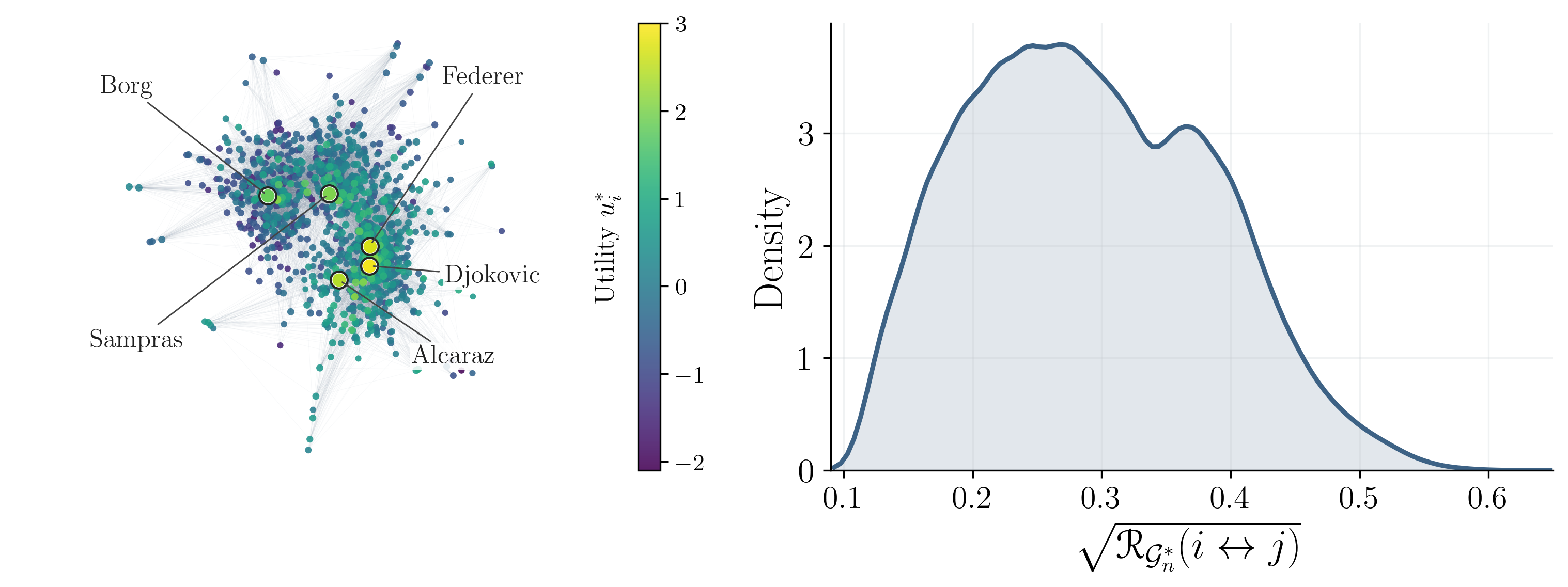}
\caption{ATP comparison network from 1980--2024. Left: Players colored by fitted utility. Right: Distribution of $\sqrt{\rr_{\G^*_n}(i\leftrightarrow j)}$ over all player pairs.}
\label{fig:tennis-era-structure}
\end{figure}

To further examine our asymptotic results, we sample 30 pairs from each of the lowest, middle, and highest fifths of the square-root effective resistance distribution and highlight three pairs from each group. In \Cref{fig:tennis-era-inference}, their mean absolute errors increase with the square-root effective resistance, qualitatively consistent with the bound in \Cref{thm:main}. We also construct the confidence intervals using \eqref{confidence_interval} for the nine highlighted pairs, including four pairs with no direct matches. Their coverage probabilities range from $0.936$ to $0.964$.

For these nine pairs, the mean absolute linearization residual is only $1.6\%$--$10.3\%$ of $\sqrt{\rr_{\G_n^*}(i\leftrightarrow j)}$, substantially smaller than the scale suggested by the reciprocal spectral gap. This is not surprising, since the spectral-gap term provides a uniform bound on the nonlinear remainder over all pairs and graphs and can therefore be conservative for a particular pair in some settings. Developing sharp pair-specific control of the nonlinear remainder could provide a sharper characterization of when the normal approximation is accurate and lead to less conservative inference. We leave this question for future work.

\begin{figure}[hbt!]
\centering
\includegraphics[width=.98\textwidth]{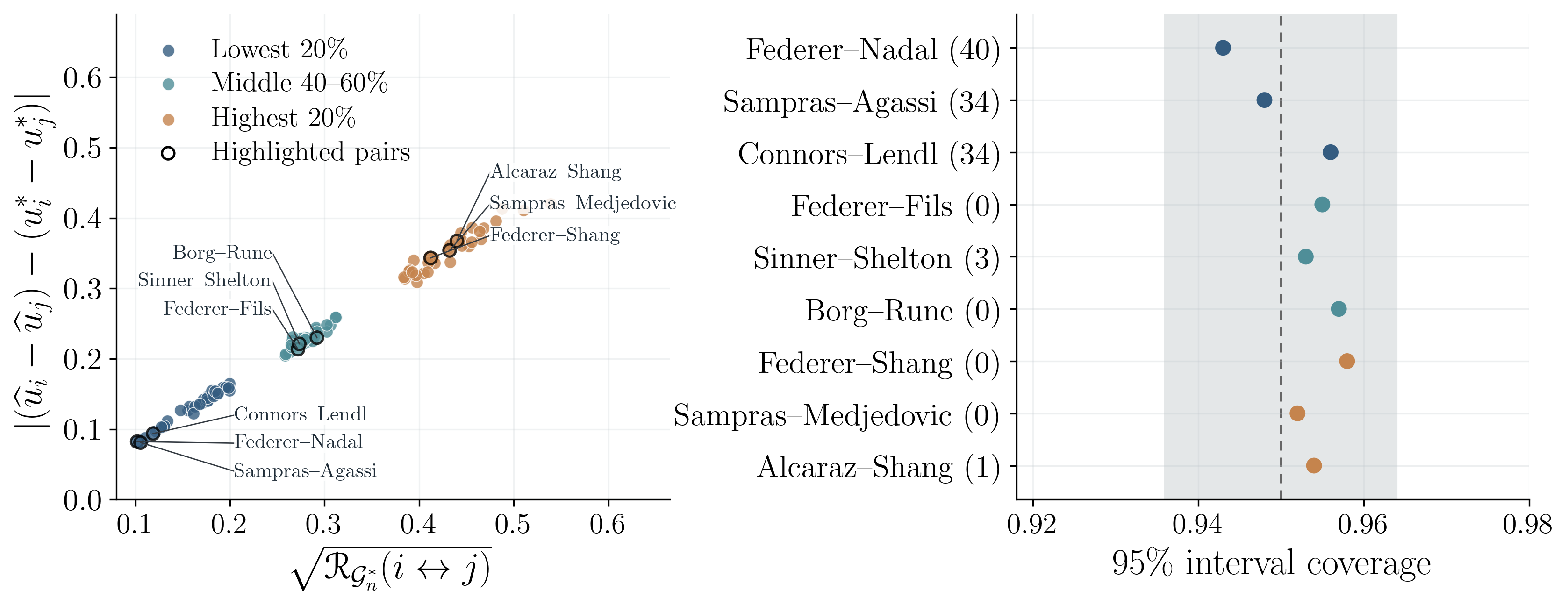}
\caption{Left: Mean absolute error for the $90$ selected pairs with the $9$ named pairs marked in black. Right: Coverage probabilities of $95\%$ confidence intervals for the named pairs, with direct match counts in parentheses. The gray band shows the $95\%$ confidence interval of mean coverage probability among all pairs.}
\label{fig:tennis-era-inference}
\end{figure}

\FloatBarrier

\section{Conclusion and future work}\label{sec:con}

This article introduces a new approach to establishing the uniform consistency and asymptotic normality of the MLE in the BT model under general comparison designs. We obtain a pairwise error bound consisting of a pair-specific term governed by the effective resistance between the pair and a nonlinear remainder controlled by the global spectral gap. Our results are tight for certain graph designs up to logarithmic factors and apply to a wide range of comparison graphs exhibiting substantial heterogeneity.

Our approach relies crucially on the electrical-network interpretation of the Hessian of the negative log-likelihood, together with certain regularity conditions, but not on the specific form of the logistic link function. It is therefore interesting to know whether our approach can be extended to more general pairwise-comparison models \cite{han2023general} or multiple comparison models such as the Plackett--Luce model \cite{MR0108411,plackett1975analysis}.  

On the other hand, our results suggest that, for inference under heterogeneous graph designs, debiasing may be necessary for pairs with particularly small effective resistance. Such pairs typically correspond to objects that are compared many times or are highly connected to the rest of the graph, and they are often of particular practical interest. Our control of the nonlinear remainder provides some insight into the resulting bias, but developing an exact debiasing procedure requires pair-specific asymptotic expansions, which we leave for future work.

\section*{Acknowledgment and AI use}
R. Han was supported by the Hong Kong Research Grants Council (No. 15302225) and the Hong Kong Polytechnic University (P0044617, P0045351). Y. Xu was partially supported by start-up funding from the University of Kentucky and NSF grant no. DMS-2607989. 

The authors acknowledge the use of generative AI in preparing this manuscript. GPT-6 was used to identify the example and specific parameter choices in \Cref{thm:bipartite}, which provide evidence for the sharpness of our error bound. GPT-6 was also used for proofreading, copy-editing, and assistance with the numerical simulations. All mathematical arguments and results presented in the manuscript were developed by the authors, who take full responsibility for their correctness.

\printbibliography

@article{baker2014dynamic,
  title={A dynamic paired comparisons model: Who is the greatest tennis player?},
  author={Baker, Rose D. and McHale, Ian G.},
  journal={Eur. J. Oper. Res.},
  volume={236},
  number={2},
  pages={677--684},
  year={2014},
  doi={10.1016/j.ejor.2013.12.028}
}

@inproceedings{hendrickx2019graph,
  title={Graph resistance and learning from pairwise comparisons},
  author={Hendrickx, Julien and Olshevsky, Alexander and Saligrama, Venkatesh},
  booktitle={ICML},
  pages={2702--2711},
  year={2019},
  organization={PMLR}
}

@inproceedings{hendrickx2020minimax,
  title={Minimax rate for learning from pairwise comparisons in the BTL model},
  author={Hendrickx, Julien and Olshevsky, Alex and Saligrama, Venkatesh},
  booktitle={ICML},
  pages={4193--4202},
  year={2020},
  organization={PMLR}
}

@inproceedings{bong2022generalized,
  title={Generalized results for the existence and consistency of the MLE in the Bradley-Terry-Luce model},
  author={Bong, Heejong and Rinaldo, Alessandro},
  booktitle={ICML},
  pages={2160--2177},
  year={2022},
  organization={PMLR}
}

@misc{sackmann_tennis_atp,
  author       = {Sackmann, Jeff},
  title        = {Tennis ATP: ATP Tennis Rankings, Results, and Stats},
  howpublished = {\url{https://github.com/JeffSackmann/tennis_atp}},
  note         = {Tennis Abstract. Licensed under CC BY-NC-SA 4.0},
  year         = {2026}
}

@book{brouwer2012spectra,
  author    = {Brouwer, Andries E. and Haemers, Willem H.},
  title     = {Spectra of Graphs},
  series    = {Universitext},
  publisher = {Springer},
  address   = {New York},
  year      = {2012}
}

@article{chen2025item,
  title={Item response theory—A statistical framework for educational and psychological measurement},
  author={Chen, Yunxiao and Li, Xiaoou and Liu, Jingchen and Ying, Zhiliang},
  journal={Stat. Sci.},
  volume={40},
  number={2},
  pages={167--194},
  year={2025},
  publisher={Institute of Mathematical Statistics}
}

@article{holland1983stochastic,
  title={Stochastic blockmodels: First steps},
  author={Holland, Paul W and Laskey, Kathryn Blackmond and Leinhardt, Samuel},
  journal={Soc. Networks},
  volume={5},
  number={2},
  pages={109--137},
  year={1983},
  publisher={Elsevier}
}

@inproceedings{lee2026entrywise,
  title={Entrywise Error Bounds for Spectral Ranking with Semi-Random Adversaries},
  author={Lee, Dongmin and Makur, Anuran and Singh, Japneet},
  booktitle={SIGKDD},
  pages={2449--2460},
  year={2026}
}

@article{hamilton2025many,
  title={The many routes to the ubiquitous Bradley--Terry model},
  author={Hamilton, Ian Allister and Tawn, Nicholas G. and Firth, David},
  journal={Stat. Sci.},
  year={2025},
  note={In press},
  issn={0883-4237}
}

@article{bozoki2016application,
  title={An application of incomplete pairwise comparison matrices for ranking top tennis players},
  author={Boz{\'o}ki, S{\'a}ndor and Csat{\'o}, L{\'a}szl{\'o} and Temesi, J{\'o}zsef},
  journal={Eur. J. Oper. Res.},
  volume={248},
  number={1},
  pages={211--218},
  year={2016},
  doi={10.1016/j.ejor.2015.06.069}
}

@incollection{mcfadden1973conditional,
  title={Conditional logit analysis of qualitative choice behavior},
  author={McFadden, Daniel},
  booktitle={Frontiers in Econometrics},
  editor={Zarembka, Paul},
  pages={105--142},
  publisher={Academic Press},
  address={New York},
  year={1974}
}

@inproceedings{Christiano2017,
  title={Deep Reinforcement Learning from Human Preferences},
  author={Christiano, Paul F. and Leike, Jan and Brown, Tom B. and Martic, Miljan and Legg, Shane and Amodei, Dario},
  booktitle={NeurIPS},
  volume={30},
  pages={4299--4307},
  year={2017}
}

@inproceedings{Rafailov2023,
  title={Direct Preference Optimization: Your Language Model Is Secretly a Reward Model},
  author={Rafailov, Rafael and Sharma, Archit and Mitchell, Eric and Ermon, Stefano and Manning, Christopher D. and Finn, Chelsea},
  booktitle={NeurIPS},
  volume={36},
  pages={53728--53741},
  year={2023}
}

@inproceedings{Sun2025,
  title={Rethinking Reward Modeling in Preference-based Large Language Model Alignment},
  author={Sun, Hao and Shen, Yunyi and Ton, Jean-Francois},
  booktitle={ICLR},
  year={2025}
}

@article{chung1998isoperimetric,
  title={Isoperimetric inequalities for cartesian products of graphs},
  author={Chung, Fan RK and Tetali, Prasad},
  journal={Comb. Probab. Comput.},
  volume={7},
  number={2},
  pages={141--148},
  year={1998},
  publisher={Cambridge University Press}
}

@incollection{mohar1993eigenvalues,
  title={Eigenvalues in Combinatorial Optimization},
  author={Mohar, Bojan and Poljak, Svatopluk},
  booktitle={Combinatorial and Graph-Theoretical Problems in Linear Algebra},
  editor={Brualdi, Richard A. and Friedland, Shmuel and Klee, Victor},
  series={IMA Volumes in Mathematics and Its Applications},
  volume={50},
  pages={107--151},
  publisher={Springer},
  year={1993}
}

@article{mohar1989isoperimetric,
  title={Isoperimetric numbers of graphs},
  author={Mohar, Bojan},
  journal={J. Comb. Theory, B.},
  volume={47},
  number={3},
  pages={274--291},
  year={1989},
  publisher={Elsevier}
}

@article{devriendt2022discrete,
  author  = {Devriendt, Karel and Lambiotte, Renaud},
  title   = {Discrete Curvature on Graphs from the Effective Resistance},
  journal = {J. Phys.: Complex},
  volume  = {3},
  number  = {2},
  pages   = {025008},
  year    = {2022},
  doi     = {10.1088/2632-072X/ac730d}
}

@book{lyons2016probability,
  author    = {Lyons, Russell and Peres, Yuval},
  title     = {Probability on Trees and Networks},
  series    = {Cambridge Series in Statistical and Probabilistic Mathematics},
  volume    = {42},
  publisher = {Cambridge University Press},
  address   = {New York},
  year      = {2016},
  pages     = {xv+699},
  isbn      = {978-1-107-16015-6},
  doi       = {10.1017/9781316672815}
}

@article{wainwright2009sharp,
  title={Sharp thresholds for High-Dimensional and noisy sparsity recovery using $\ell_1$-Constrained Quadratic Programming (Lasso)},
  author={Wainwright, Martin J},
  journal={IEEE Trans. Inf. Theory},
  volume={55},
  number={5},
  pages={2183--2202},
  year={2009},
  publisher={IEEE}
}

@article{fang2026recent,
  title={Recent advances in the Bradley--Terry Model: theory, algorithms, and applications},
  author={Fang, Shuxing and Han, Ruijian and Luo, Yuanhang and Xu, Yiming},
  journal={arXiv preprint arXiv:2601.14727},
  year={2026}
}

@article{dong2026statistical,
  title={Statistical ranking with dynamic covariates},
  author={Dong, Pinjun and Han, Ruijian and Jiang, Binyan and Xu, Yiming},
  journal={J. R. Stat. Soc., B: Stat. Methodol},
  volume={88},
  number={1},
  pages={221--238},
  year={2026},
  publisher={Oxford University Press UK}
}

@inproceedings{yang2024top,
  title={Top-$ K $ ranking with a monotone adversary},
  author={Yang, Yuepeng and Chen, Antares and Orecchia, Lorenzo and Ma, Cong},
  booktitle={COLT},
  pages={5123--5162},
  year={2024},
  organization={PMLR}
}

@inproceedings{chandra1989electrical,
  title={The electrical resistance of a graph captures its commute and cover times},
  author={Chandra, Ashok K and Raghavan, Prabhakar and Ruzzo, Walter L and Smolensky, Roman},
  booktitle={STOC},
  pages={574--586},
  year={1989}
}

@article{chen2023ranking,
  title={Ranking from pairwise comparisons in general graphs and graphs with locality},
  author={Chen, Yanxi},
  journal={arXiv preprint arXiv:2304.06821},
  year={2023}
}

@article{chen2022partial,
  title={Partial recovery for top-k ranking: optimality of MLE and suboptimality of the spectral method},
  author={Chen, Pinhan and Gao, Chao and Zhang, Anderson Y},
  journal={Ann. Statist.},
  fjournal={The Annals of Statistics},
  volume={50},
  number={3},
  pages={1618--1652},
  year={2022},
  publisher={Institute of Mathematical Statistics}
}

@article{han2024statistical,
  title={Statistical inference for pairwise comparison models},
  author={Han, Ruijian and Tang, Wenlu and Xu, Yiming},
  journal={arXiv preprint arXiv:2401.08463},
  year={2024}
}

@article{fan2023spectral,
  title={Spectral ranking inferences based on general multiway comparisons},
  author={Fan, Jianqing and Lou, Zhipeng and Wang, Weichen and Yu, Mengxin},
  journal={Oper. Res.},
  volume={74},
  number={1},
  pages={161--180},
  year={2025},
  publisher={INFORMS}
}

@article{han2026convergence,
  title={Convergence analysis of a family of Zermelo-type iterations for the Bradley--Terry model},
  author={Han, Ruijian and Lu, Ding and Xu, Yiming},
  journal={arXiv preprint arXiv:2607.22221},
  year={2026}
}

@article{han2023unified,
  title={A unified analysis of likelihood-based estimators in the {P}lackett--{L}uce model},
  author={Han, Ruijian and Xu, Yiming},
  journal={Ann. Statist.},
  volume={53},
  number={5},
  pages={2077--2102},
  year={2025},
  doi={10.1214/25-AOS2530}
}

@article{von2014hitting,
	AUTHOR = {von Luxburg, Ulrike and Radl, Agnes and Hein, Matthias},
	TITLE = {Hitting and commute times in large random neighborhood graphs},
	JOURNAL = {J. Mach. Learn. Res.},
	FJOURNAL = {Journal of Machine Learning Research},
	VOLUME = {15},
	YEAR = {2014},
	PAGES = {1751--1798},
	ISSN = {1532-4435},
	MRCLASS = {05C80 (60C05)},
	MRNUMBER = {3225242},
	MRREVIEWER = {Anna Torriero},
}

@article{gao2023uncertainty,
	AUTHOR = {Gao, Chao and Shen, Yandi and Zhang, Anderson Y.},
	TITLE = {Uncertainty quantification in the {B}radley-{T}erry-{L}uce
	model},
	JOURNAL = {Inf. Inference},
	FJOURNAL = {Information and Inference. A Journal of the IMA},
	VOLUME = {12},
	YEAR = {2023},
	NUMBER = {2},
	PAGES = {1073--1140},
	ISSN = {2049-8764},
	MRCLASS = {62F07 (62F12 62F25 62H12)},
	MRNUMBER = {4565758},
}

@article{han2023general,
  title={A general pairwise comparison model for extremely sparse networks},
  author={Han, Ruijian and Xu, Yiming and Chen, Kani},
   JOURNAL = {J. Amer. Statist. Assoc.},
	FJOURNAL = {Journal of the American Statistical Association},
    volume={118},
  number={544},
  pages={2422--2432},
  year={2023},
  publisher={Taylor \& Francis}
}

@article {MR0070925,
	AUTHOR = {Bradley, Ralph Allan and Terry, Milton E.},
	TITLE = {Rank analysis of incomplete block designs. {I}. {T}he method
	of paired comparisons},
	JOURNAL = {Biometrika},
	FJOURNAL = {Biometrika},
	VOLUME = {39},
	YEAR = {1952},
	PAGES = {324--345},
	ISSN = {0006-3444},
	MRCLASS = {62.0X},
	MRNUMBER = {0070925},
	MRREVIEWER = {W. S. Connor},
}

@article{thurstone1927law,
	title={A law of comparative judgment.},
	author={Thurstone, Louis L},
	journal={Psychol. Rev.},
	volume={34},
	number={4},
	pages={273--286},
	year={1927},
	publisher={Psychological Review Company}
}

@article{MR1724040,
    AUTHOR = {Simons, Gordon and Yao, Yi-Ching},
TITLE = {Asymptotics when the number of parameters tends to infinity in
the {B}radley-{T}erry model for paired comparisons},
JOURNAL = {Ann. Statist.},
FJOURNAL = {The Annals of Statistics},
  volume={27},
  number={3},
  pages={1041--1060},
  year={1999},
  publisher={Institute of Mathematical Statistics},
  MRCLASS = {62J15 (62E20 62F05 62F12)},
  MRNUMBER = {1724040}
}

@article{MR2987494,
    AUTHOR = {Yan, Ting and Yang, Yaning and Xu, Jinfeng},
TITLE = {Sparse paired comparisons in the {B}radley-{T}erry model},
JOURNAL = {Statist. Sinica},
FJOURNAL = {Statistica Sinica},
VOLUME = {22},
YEAR = {2012},
NUMBER = {3},
PAGES = {1305--1318},
ISSN = {1017-0405},
MRCLASS = {62J15},
MRNUMBER = {2987494}
}

@article {MR3953449,
	AUTHOR = {Chen, Yuxin and Fan, Jianqing and Ma, Cong and Wang, Kaizheng},
	TITLE = {Spectral method and regularized {MLE} are both optimal for
	top-{$K$} ranking},
	JOURNAL = {Ann. Statist.},
	FJOURNAL = {The Annals of Statistics},
	VOLUME = {47},
	YEAR = {2019},
	NUMBER = {4},
	PAGES = {2204--2235},
	ISSN = {0090-5364},
	MRCLASS = {62F07 (60J20 62B10)},
	MRNUMBER = {3953449}
}

@article{han2020asymptotic,
    AUTHOR = {Han, Ruijian and Ye, Rougang and Tan, Chunxi and Chen, Kani},
TITLE = {Asymptotic theory of sparse {B}radley-{T}erry model},
JOURNAL = {Ann. Appl. Probab.},
FJOURNAL = {The Annals of Applied Probability},
VOLUME = {30},
YEAR = {2020},
NUMBER = {5},
PAGES = {2491--2515},
ISSN = {1050-5164},
MRCLASS = {60F05 (62E20 62F12 62J15)},
MRNUMBER = {4149535}
}

@article{MR3504618,
      AUTHOR = {Shah, Nihar B. and Balakrishnan, Sivaraman and Bradley, Joseph
  and Parekh, Abhay and Ramchandran, Kannan and Wainwright,
  Martin J.},
  TITLE = {Estimation from pairwise comparisons: sharp minimax bounds
  with topology dependence},
  JOURNAL = {J. Mach. Learn. Res.},
  FJOURNAL = {Journal of Machine Learning Research (JMLR)},
  VOLUME = {17},
  YEAR = {2016},
  PAGES = {1--47},
  ISSN = {1532-4435},
  MRCLASS = {62J15 (62C20 62F07)},
  MRNUMBER = {3504618}

}

@article {MR1545015,
	AUTHOR = {Zermelo, E.},
	TITLE = {Die {B}erechnung der {T}urnier-{E}rgebnisse als ein
	{M}aximumproblem der {W}ahrscheinlichkeitsrechnung},
	JOURNAL = {Math. Z.},
	FJOURNAL = {Mathematische Zeitschrift},
	VOLUME = {29},
	YEAR = {1929},
	NUMBER = {1},
	PAGES = {436--460},
	ISSN = {0025-5874},
	MRCLASS = {DML},
	MRNUMBER = {1545015},
}

@article {MR0097876,
	AUTHOR = {Ford, Jr., L. R.},
	TITLE = {Solution of a ranking problem from binary comparisons},
	JOURNAL = {Amer. Math. Monthly},
	FJOURNAL = {The American Mathematical Monthly},
	VOLUME = {64},
	YEAR = {1957},
	NUMBER = {8},
	PAGES = {28--33},
	ISSN = {0002-9890},
	MRCLASS = {62.00},
	MRNUMBER = {0097876},
	MRREVIEWER = {I. R. Savage},
}

@article {MR0125031,
	AUTHOR = {Erd\H{o}s, P. and R\'enyi, A.},
	TITLE = {On the evolution of random graphs},
	JOURNAL = {Magyar Tud. Akad. Mat. Kutat\'{o} Int. K\"{o}zl.},
	FJOURNAL = {A Magyar Tudom\'{a}nyos Akad\'{e}mia. Matematikai Kutat\'{o} Int\'{e}zet\'{e}nek
	K\"{o}zlem\'{e}nyei},
	VOLUME = {5},
	YEAR = {1960},
	PAGES = {17--61},
	ISSN = {0541-9514},
	MRCLASS = {05.40},
	MRNUMBER = {0125031},
	MRREVIEWER = {John Riordan},
}

@book{vershynin2018high,
  title={High-dimensional probability: An introduction with applications in data science},
  author={Vershynin, Roman},
  volume={47},
  year={2018},
  publisher={Cambridge University Press}
}

@article {plackett1975analysis,
	AUTHOR = {Plackett, R. L.},
	TITLE = {The analysis of permutations},
	JOURNAL = {J. Roy. Statist. Soc. Ser. C},
	FJOURNAL = {Journal of the Royal Statistical Society. Series C. Applied
	Statistics},
	VOLUME = {24},
	YEAR = {1975},
	NUMBER = {2},
	PAGES = {193--202},
	ISSN = {0035-9254},
	MRCLASS = {62E10 (62P25)},
	MRNUMBER = {391338},
}

@book {MR0108411,
	AUTHOR = {Luce, R. Duncan},
	TITLE = {Individual choice behavior: {A} theoretical analysis},
	PUBLISHER = {John Wiley \& Sons, Inc., New York; Chapman \& Hall Ltd.
	London},
	YEAR = {1959},
	PAGES = {xii+153},
	MRREVIEWER = {L. J. Savage},
}

@inproceedings{pmlr-v180-li22g,
  title     = {$\ell_{\infty}$-Bounds of the {MLE} in the {BTL} Model under General Comparison Graphs},
  author    = {Li, Wanshan and Shrotriya, Shamindra and Rinaldo, Alessandro},
  booktitle = {UAI},
  volume    = {180},
  pages     = {1178--1187},
  year      = {2022},
  publisher = {PMLR},
  url       = {https://proceedings.mlr.press/v180/li22g.html}
}

\end{document}